\documentclass[a4paper,12pt]{amsart}

\usepackage{tikz}
\usetikzlibrary{positioning,arrows,cd}

\usepackage{amsmath,amssymb,amsthm}

\usepackage[shortlabels]{enumitem}

\usepackage{verbatim}

\usepackage{aliascnt}
\usepackage[pdfpagelabels]{hyperref}

\newtheorem{thmx}{Theorem}

\newaliascnt{corx}{thmx}

\aliascntresetthe{corx}

\newtheorem{theorem}{Theorem}[section]

\newaliascnt{proposition}{theorem}
\newtheorem{proposition}[proposition]{Proposition}
\aliascntresetthe{proposition}

\newaliascnt{lemma}{theorem}
\newtheorem{lemma}[lemma]{Lemma}
\aliascntresetthe{lemma}

\newaliascnt{corollary}{theorem}
\newtheorem{corollary}[corollary]{Corollary}
\aliascntresetthe{corollary}

\newaliascnt{question}{theorem}
	
\aliascntresetthe{question}

\newaliascnt{conjecture}{theorem}
\newtheorem{conjecture}[conjecture]{Conjecture}	
\aliascntresetthe{conjecture}

\theoremstyle{definition}
\newaliascnt{definition}{theorem}
\newtheorem{definition}[definition]{Definition}
\aliascntresetthe{definition}

\theoremstyle{remark}

\newaliascnt{example}{theorem}
\newtheorem{example}[example]{Example}	
\aliascntresetthe{example}

\newaliascnt{remark}{theorem}
\newtheorem{remark}[remark]{Remark}	
\aliascntresetthe{remark}

\def\equationautorefname~#1\null{Equation~(#1)\null}

\newcommand{\Aut}{\operatorname{Aut}}

\newcommand{\Z}{{\mathbb Z}}
\newcommand{\K}{{\mathbb K}}

\newcommand{\G}{\mathcal{G}}

\renewcommand{\b}{\operatorname{b}}
\newcommand{\rad}{\operatorname{rad}}

\makeatletter
\newcommand*{\udot}{{\mathpalette\ud@t\relax}} 
\newcommand*{\ud@t}[2]{%
   \sbox\z@{\m@th$#1.$}%
   \sbox\tw@{$#1:$}%
   \raise\dimexpr\ht\tw@-\ht\z@\relax\box\z@
}
\makeatother

\DeclareMathOperator{\lcm}{lcm}
\DeclareMathOperator{\evol}{Evol}
\DeclareMathOperator{\nevol}{NEvol}
\newcommand{\Evol}{\evol^\mathrm{nd}}
\newcommand{\NEvol}{\nevol^\mathrm{nd}}

\newcommand{\iso}{\cong}
\providecommand{\Cyc}{\mathcal{R}}

\title[On the Grothendieck ring of evolution algebras]
{On the Grothendieck Ring of Finite-Dimensional Non-Degenerate Evolution Algebras}

\author[I.\ Alshatnawi]{Idrees Alshatnawi}
\address{Departamento de \'Algebra, Geometr{\'\i}a y Topolog{\'\i}a, Universidad de M{\'a}laga, 29071 M{\'a}laga, Spain}
\email{idrees.kh.alshatnawi@uma.es} 

\author[C. Costoya]{Cristina Costoya}
\address{CITMAga, Departamento de Matem{á}ticas, Universidade de Santiago de Compostela, 15782-Santiago de Compostela, Spain.}
\email{cristina.costoya@usc.es}

\author[A.\ Viruel]{Antonio Viruel}
\address{Departamento de \'Algebra, Geometr{\'\i}a y Topolog{\'\i}a, Universidad de M{\'a}laga, 29071 M{\'a}laga, Spain}
\email{viruel@uma.es}
\thanks{This work was partially supported by MCIN/AEI/10.13039/501100011033 [PID2023-149804NB-I00 to C.C., and A.V].}

\begin{document}

\begin{abstract} We study the Grothendieck ring of finite-dimensional non-degenerate evolution algebras over a field $\K$, with addition induced by direct sum and multiplication induced by tensor product. Although its underlying abelian group is freely generated by indecomposable isomorphism classes, the ring structure is much smaller: tensor products create systematic non-cancellation phenomena and many zero-divisors. The key invariant is the balance of the directed graph associated with a natural basis. We prove that, for non-degenerate evolution algebras, this balance is independent of the chosen natural basis. We then show that balance controls cancellation after tensoring: indecomposable factors of balance $1$ cancel, whereas factors of larger balance produce zero-divisor relations under mild hypotheses. Finally, we analyze the subring generated by cyclic evolution algebras and prove the predicted zero-divisor criterion for all finite direct sums of cyclic algebras.\end{abstract}

\maketitle

\section{Introduction}
Evolution algebras form a class of commutative, generally non-associative algebras designed to encode non-Mendelian inheritance. Their defining feature is the existence of a natural basis $B=\{e_1,\ldots,e_n\}$ in which mixed products vanish, $e_i e_j=0$ for $i\ne j$. Consequently, the whole multiplication is encoded by the squares of the basis elements, or equivalently by a single structure matrix. This places evolution algebras at the intersection of non-associative algebra, weighted directed graphs, and Markov-chain dynamics \cite{TianVojtechovsky2006,Tian2008,ElduqueLabra2015,ElduqueLabra2019}.

A central problem in the theory is classification, usually understood as the determination of the isomorphism types of indecomposable evolution algebras. Although low-dimensional classifications are known \cite{CabreraSilesVelasco2017}, the case-by-case approach quickly becomes unfeasible. In fact, realization results show that every finite group occurs as the full automorphism group of a non-degenerate, finite-dimensional indecomposable evolution algebra \cite{CostoyaLigourasTocinoViruel2022, CostoyaMunozTocinoViruel2023}. Thus, the classification problem for evolution algebras is at least as complicated as the classification of finite groups, and cannot reasonably be expected to admit a simple complete invariant. This obstruction motivates the categorical viewpoint adopted in this paper: we study how evolution algebras combine, rather than trying to enumerate them one by one.

The category $\Evol_\K$ of finite-dimensional non-degenerate
evolution algebras over a field $\K$ is closed under two operations:
the direct sum $\oplus$, which is indeed the categorical product in $\Evol_\K$, and the tensor
product $\otimes$. Both are symmetric monoidal, and $\otimes$ is
distributive with respect to $\oplus$, see \autoref{sec:monoidal}. This is exactly the input needed
to form a \emph{Grothendieck ring} $\G(\Evol_\K)$, in which isomorphism
classes are added by $\oplus$ and multiplied by $\otimes$, see \autoref{sec:grothendieck}.

The pay-off is a change of perspective on what ``classification'' should
mean. As an abelian group, $(\G(\Evol_\K),+)$ is free on the isomorphism
classes of non-degenerate indecomposable algebras, so describing its \emph{additive} generators
is just the original, intractable problem in disguise. As a ring,
however, $\G(\Evol_\K)$ is ``smaller"; there exist indecomposable non-degenerate evolution algebras which are tensorially decomposable, see \autoref{sec:decomposition_tensor}, and moreover $\G(\Evol_\K)$ has many zero-divisors, see \autoref{thm:zero-divisors} below. Therefore, we propose to classify \emph{ring} generators of $\G(\Evol_\K)$: these are some of the additive generators but also exploit the multiplication, and
there are far fewer of them. 

Understanding $\G(\Evol_\K)$ as a ring, rather than merely as an abelian
group, reframes the classification of evolution algebras. A description of
its ring generators and of its ideal of zero-divisors would, in
particular, open the door to a \emph{probabilistic} classification (asking
for the typical behaviour of a random evolution algebra rather than a
complete list), and to the study of growth, that is, the asymptotics of the
number of indecomposables of a given dimension. 

We conclude this section by proving that the ring \(\G(\Evol_\K)\) contains zero-divisors. This introductory argument also serves as a preview of the methods developed in the sequel, where the relevant constructions and proofs are given in full detail. The central invariant in the arguments is the balance of the directed graph associated with a natural basis. While the graph itself may depend on the chosen basis, its balance does not: for non-degenerate evolution algebras it is an intrinsic invariant (\autoref{thm:balance_invariant}).

\begin{thmx}\label{thm:zero-divisors}
    Let $X$ be a finite dimensional non-degenerate evolution $\K$-algebra, and let
    $$X=\bigoplus_{j=1}^r X_j$$
    be its optimal direct-sum decomposition into indecomposable evolution ideals. Assume each $X_j$ has balance $$\b(X_j)>1.$$ Then $X$ gives rise to a zero-divisor in $\G(\Evol_\K)$, that is, there exist non-degenerate finite-dimensional evolution $\K$-algebras $A$ and $B$ such that $A\otimes X\cong B\otimes X$, while $A\not\cong B$.
\end{thmx}
\begin{proof}
For each $j$, fix a prime $p_j$ dividing $\b(X_j)$, and let $U_j=C_{p_j}$, $V_j=\K^{\oplus p_j}$. Observe that since $p_j>1$, $U_j\not\cong V_j$ as they have different optimal direct-sum decomposition into indecomposable evolution ideals. But the $p_j$-isomorphism in
\autoref{ex:miso-cyclic} shows $U_j^{\Phi}=V_j$, so by
\autoref{thm:cancel-Cn} one has $U_j\otimes C_{p_j}\iso V_j\otimes C_{p_j}$ relative to $C_{p_j}$. Applying
\autoref{thm:balance-prime} with $D=X_j$ gives
$U_j\otimes X_j\iso V_j\otimes X_j$, so the isomorphism class of $X_j$, namely $[X_j]$, is a
zero-divisor in $\G(\Evol_\K)$ since $$[X_j]([C_{p_j}]-p_j[\K])=0.$$ 

Define $q=\rad(\prod_{j=1}^r p_j)$, that is, $q$ is the largest square-free integer that divides $\prod_{j=1}^r p_j$, and let $q=q_1\ldots q_s$ be its prime factorization. Therefore, for each $j$ there exists some $q_i$ such that $p_j=q_i$ and $[X_j]([C_{q_i}]-q_i[\K])=0,$ and
\begin{align*}
    [X]\big(\prod_{i=1}^s([C_{q_i}]-q_i[\K])\big) &= (\sum_{j=1}^r[X_j])\big(\prod_{i=1}^s([C_{q_i}]-q_i[\K])\big)\\
    &= \sum_{j=1}^r[X_j]\big(\prod_{i=1}^s([C_{q_i}]-q_i[\K])\big)\\
    &=0
\end{align*}
Now, grouping the summands according to sign as it is done in \autoref{lem:lovasz}, we get
$$\prod_{i=1}^s([C_{q_i}]-q_i[\K])=[A]-[B]$$
thus $X\otimes A\cong X\otimes B$. Notice that $A\not\cong B$, or equivalently $[A]-[B]\ne 0,$ since by repeated
use of \autoref{lem:CnCm} we get that $A$ and $B$ have different optimal direct-sum decomposition into indecomposable evolution ideals. In particular $A$ contains a summand of type $C_q\cong{\otimes}_{i=1}^s C_{q_i}$ while $B$ does not.
\end{proof}

While we do not have a complete description of all zero-divisors in $\G(\Evol_\K)$, the following conjecture appears to be well motivated.

\begin{conjecture}\label{conj:balance}
An evolution $\K$-algebra is a zero-divisor in $\G(\Evol_\K)$ if and only if
each of its indecomposable direct summands has balance greater than $1$.
\end{conjecture}

\begin{remark}\label{rem:reduce-indec}
Notice that \autoref{lem:lovasz} allows one to \emph{combine} the zero-divisor relations associated with the indecomposable summands into a single relation for their direct sum. This establishes precisely the sufficiency direction of \autoref{conj:balance}, which \autoref{thm:zero-divisors} already proves for arbitrary direct sums.

By contrast, \autoref{lem:lovasz} does not reduce the necessity direction to the indecomposable case. There is no converse ``de-combination'' principle: a direct sum containing both a balance-$1$ summand and summands of balance greater than $1$ could, a priori, be a zero-divisor for reasons that are not detected at the level of the individual summands. Therefore, \autoref{conj:balance} cannot, on the basis of \autoref{lem:lovasz}, be reduced to the indecomposable case.
\end{remark}

\section{Graphs and the balance invariant}\label{sec:graphs}

In this section we briefly recall the basics in Graph Theory we use in the following sections. All graphs in this paper are directed, and we follow the conventions of \cite{ElduqueLabra2015,ElduqueLabra2019,Lovasz}.

\begin{definition}
A \emph{graph} is a pair
$\Gamma=(V,E)$ consisting of a finite set of vertices $V$ and a set of
edges (or arrows) $E\subseteq V\times V$. 
\end{definition}

\begin{definition}\label{def:balance}
A \emph{path} in $\Gamma$ is a sequence
$$
  \gamma=(v_0,e_1,v_1,\dots,v_{n-1},e_n,v_n),\qquad n\ge 0,
$$
with $v_0,\dots,v_n\in V$ and $e_1,\dots,e_n\in E$ such that, for each $i\in\{1,\dots,n\}$, either $e_i=(v_{i-1},v_i)$ (the edge is traversed \emph{forwards}) or $e_i=(v_i,v_{i-1})$ (traversed \emph{backwards}). A (non-simple) \emph{cycle} is a path with $v_0=v_n$. The \emph{balance} of $\gamma$ is the integer
$$
\b(\gamma)=\#\{\,i : e_i=(v_{i-1},v_i)\,\}-\#\{\,i : e_i=(v_i,v_{i-1})\,\},
$$
that is, $+1$ for each edge traversed in the ``right'' direction (from
$v_0$ towards $v_n$) and $-1$ for each edge traversed in the ``wrong''
direction, summed over $i$.
\end{definition}

\begin{remark}\label{rm:balance_vs_concatenation}
Observe that balance behaves naturally with respect to the two elementary operations on paths: it is additive under concatenation and changes sign under reversal. Thus, if $\alpha$ and $\beta$ are paths in $\Gamma$ with the same final vertex, then $\alpha\star\beta^{-1}$ is a well-defined path, where $\star$ denotes concatenation and $\beta^{-1}$ is the path $\beta$ traversed in the opposite direction. Hence
$$\b(\alpha\star\beta^{-1})=\b(\alpha)+\b(\beta^{-1}) =\b(\alpha)-\b(\beta).$$
\end{remark}

\begin{definition}\label{def:graphbalance}
The \emph{balance} of $\Gamma$ is the greatest common divisor of the
absolute values of the balances of its cycles,
\[
  \b(\Gamma)=\gcd\{\,|\b(\gamma)| : \gamma \text{ a cycle in } \Gamma\,\}.
\]
\end{definition}

\begin{remark}\label{rm:balance=weight}
The notions of balance of a path and of a graph introduced here agree with those in \cite{ElduqueLabra2015,ElduqueLabra2019} and constitute the terminology commonly used in the context of evolution algebras. They also coincide with the notions of weight introduced by Chen and Chen in \cite{ChenChen}.
\end{remark}

\begin{definition}\label{def:degrees}
The \emph{indegree} and \emph{outdegree} of a vertex $v\in V$ are
$$
  \deg^-(v)=\#\{\,w\in V : (w,v)\in E\,\},
  \qquad
  \deg^+(v)=\#\{\,w\in V : (v,w)\in E\,\}.
$$
The vertex $v$ is a \emph{source} if $\deg^-(v)=0$ and a \emph{sink} if
$\deg^+(v)=0$. The graph $\Gamma$ is \emph{connected} if its underlying
undirected graph is connected, that is, if for every $v,w\in V$ there is a
path from $v$ to $w$.
\end{definition}

The following results characterize the balance of connected graphs.

\begin{lemma}\label{lem:balance-one-cycle}
Let $\Gamma=(V,E)$ be a connected graph with $\b(\Gamma)=d$. Then, for every $v_0\in V$,
$\Gamma$ contains a cycle $\gamma$ with base point $v_0$ and $\b(\gamma)=d$.
\end{lemma}
\begin{proof}
By definition $\b(\Gamma)=\gcd\{\,|\b(\gamma)|:\gamma\text{ a cycle}\,\}=d$, so by Bézout's identity, there
are cycles $\gamma_1,\dots,\gamma_r$ and integers $\alpha_1,\dots,\alpha_r$
with $\sum_{j=1}^r\alpha_j\,\b(\gamma_j)=d$.  Fix the base point $v_0$. Since $\Gamma$ is connected, for
each $j\geq 1$, there is a path $\beta_j$ from $v_0$ to the base point of $\gamma_j$.
Concatenating $\beta_j$, traversing each $\gamma_j$ exactly $|\alpha_j|$ times in the
direction $\operatorname{sg}(\alpha_j)$ and returning along
$\beta_j^{-1}$ for each $j\ge1$, namely
$$
  \gamma=\operatornamewithlimits{\star}_{j=1}^r \bigl(\beta_j\star\gamma_j^{\alpha_j}\star\beta_j^{-1}\bigr),
$$
gives a cycle based at $v_0$. Applying \autoref{rm:balance_vs_concatenation} yields
$$\b(\gamma)=\sum_{j=1}^r\big(\b(\beta_j)+\alpha_j\b(\gamma_j)-\b(\beta_j)\big)=\sum_{j=1}^r\alpha_j\b(\gamma_j)=d. \qedhere
$$
\end{proof}

\begin{lemma}\label{lem:balance_by_grading}
Let $\Gamma=(V,E)$ be a connected graph with no sinks. Then
$\b(\Gamma)\ge 1$, and for any positive integer $d$ the following are equivalent:
\begin{enumerate}[label={\rm (\roman{*})}]
\item\label{lem:balance_by_grading_i} $d$ divides $\b(\Gamma)$.
\item\label{lem:balance_by_grading_ii} There exists a surjective map $$h_d\colon V \longrightarrow \Z/d,$$ such that if $(v,w)\in E$ then $h_d(w)=h_d(v)+1 \pmod d.$
\end{enumerate}
\end{lemma}
\begin{proof}
Since $\Gamma$ has no sink, every vertex has out-degree $\ge 1$, and a finite graph with this property
contains a cycle, of balance equal to its length, so
$\b(\Gamma)\ge 1$. 

Fix now an integer $d>0.$

We assume first that $d$ divides $\b(\Gamma)$. Choose a vertex $v_0$, and define $h_d(v_0)=0$. For any vertex $v$,
choose a path $\gamma$ from $v_0$ to $v$ and set
$h_d(v)=\b(\gamma)\bmod d$. This is well defined: a second path
$\gamma'$ gives a cycle $\gamma-\gamma'$, whose balance is divisible by
$\b(\Gamma)$ and hence by $d$, so $\b(\gamma)\equiv \b(\gamma')\pmod d$.  
Appending an edge $(v,w)\in E$ to the path $\gamma$ gives a path from $v_0$ to $w$ that raises balance by $1$, giving
$h_d(w)=h_d(v)+1$. Finally, since $\Gamma$ has no sinks it contains a directed cycle $Z$ of some length $L$; its balance is $L$, so $d\mid\b(\Gamma)\mid L$ and $L\ge d$. Traversing $Z$, consecutive vertices have consecutive heights $h_d(u_0)+k \pmod d$, and $L\ge d$ forces these to exhaust $\Z/d$; hence $h_d$ is surjective.

Now, assume there exists 
$h_d\colon V\to\Z/d$ such that if $(v,w)\in E$ then $h_d(w)=h_d(v)+1 \pmod d.$ For any given vertex $v_0\in V$ traversing the directed cycle with base point $v_0$ given by \autoref{lem:balance-one-cycle} shows that $h_d(v_0)=h_d(v_0)+\b(\Gamma) \pmod d,$ that holds only if $\b(\Gamma)=0 \pmod d,$ that is, if $d$ divides $\b(\Gamma).$
\end{proof}

\section{Basic definitions on evolution algebras}\label{sec:prelim}

In this section, we establish the notation and recall the basic definitions concerning evolution algebras that are used throughout the paper. Unless stated otherwise, $\K$ denotes a field, and all $\K$-algebras are assumed to be finite-dimensional.

\begin{definition}\label{def:evol}
A $\K$-algebra $X$ is an \emph{evolution algebra} if it admits a basis
$B=\{e_1,\dots,e_n\}$, called a \emph{natural basis}, such that
$e_i\, e_j = 0$ whenever $i\neq j$.

Relative to such a basis the products of the basis elements with
themselves are recorded by scalars $w_{ik}\in\K$ via
\[
  e_i^{\,2} \;=\; \sum_{k=1}^{n} w_{ik}\, e_k ,
\]
and the matrix $M_B(X)=(w_{ik})$ is the \emph{structure matrix} of $X$
relative to $B$.
\end{definition}

Evolution algebras are commutative because the product of distinct basis vectors is zero. However, they are generally neither associative nor power-associative. Once a natural basis $B$ is fixed, the algebra is completely determined by its structure matrix $M_B(X)=(w_{ik})$. Geometrically and biologically, this structure carries two primary interpretations \cite{Tian2008}: in genetics, each generator $e_i$ represents a type and $e_i^2$ its offspring distribution, while the Markov-chain interpretation takes the $w_{ik}$ to be transition weights from state $i$ to state $k$.

\begin{remark}\label{rem:basisnotunique}
The natural basis is not unique, and neither is the structure matrix: a given evolution algebra can carry several genuinely different natural bases. This non-uniqueness is the source of most of the subtlety in the theory and is precisely what the rigidified category of \autoref{sec:nevol} is designed to control.
\end{remark}

\begin{definition}
An evolution $\K$-algebra $X$ is \emph{non-degenerate} if it admits a natural basis $B=\{e_1,\dots,e_n\}$ such that $e_i^2\ne 0$ for $i=1,\ldots,n.$
\end{definition}

\begin{remark}\label{rem:degenerate}
The property of being non-degenerate is independent of the chosen basis. Indeed, according to \cite[Corollary 2.19]{CabreraSilesVelasco2016}, an evolution algebra $X$ is non-degenerate if and only if
$$\operatorname{Ann}(X)=\{x\in X : xX=0\}=0.$$    
\end{remark}

\begin{definition}\label{def:digraph}
Let $X$ be an evolution algebra with natural basis $B=\{e_1,\dots,e_n\}$ and structure matrix $(w_{ik})$. The \emph{associated graph} $\Gamma(X,B)$ is defined as the graph on the vertex set $\{1,\dots,n\}$ that has an edge $(i,k)$  whenever $w_{ik}\neq 0$ \cite{ElduqueLabra2015,ElduqueLabra2019}.
\end{definition}

Many algebraic features of $X$ (nilpotency, simplicity, ideals) are visible in $\Gamma(X,B)$. However, the correspondence between evolution algebras and their associated graphs is neither bijective nor functorial: it depends on the chosen natural basis (\autoref{rem:basisnotunique}), and a homomorphism of evolution algebras need not rescale the prescribed natural bases, hence need not induce a morphism of graphs. Nevertheless, we prove in \autoref{sec:balance} that the balance of the graph $\Gamma(X,B)$ is an intrinsic invariant of $X$ that does not depend on the chosen natural basis $B$ if $X$ is non-degenerate.

\section{Monoidal structure on $\Evol_\K$}\label{sec:monoidal}

Recall the category $\Evol_\K$ is defined as the full subcategory of the category of $\K$-algebras which is spanned by  finite-dimensional non-degenerate evolution algebras. The category $\Evol_\K$ is closed under two operations:

\begin{itemize}
\item The \emph{direct sum} $X\oplus Y$: It is the \emph{categorical
product} in $\Evol_\K$. Indeed, given $X$, and $Y$ non-degenerate evolution algebras with natural basis
$B$ and $B'$ respectively, then $B\oplus\{0\}\sqcup \{0\}\oplus B'$ is a natural basis of $X\oplus Y$, whose
structure matrix is the block diagonal matrix of those of $X$ and $Y$. Moreover, since multiplication in $X\oplus Y$ is component wise, the square of every element in $B\oplus\{0\}\sqcup \{0\}\oplus B'$ is non trivial, that is $X\oplus Y$ is non-degenerate. Moreover, $X\oplus Y$ is the categorical product: the projections
$\pi_X\colon X\oplus Y\to X$ and $\pi_Y\colon X\oplus Y\to Y$ are
homomorphisms of evolution algebras, and for any non-degenerate evolution
$\K$-algebra $Z$ with homomorphisms $f\colon Z\to X$ and $g\colon Z\to Y$,
the map $z\mapsto(f(z),g(z))$ is the unique homomorphism $Z\to X\oplus Y$
with $\pi_X\circ(f,g)=f$ and $\pi_Y\circ(f,g)=g$; it is multiplicative
because products across the two blocks vanish. 

Every non-degenerate evolution algebra admits a unique optimal direct-sum decomposition into indecomposable evolution ideals \cite[Theorem 5.11]{CabreraSilesVelasco2016}. In other words, every non-degenerate evolution algebra decomposes in a  unique way as a direct sum of indecomposable evolution algebras. Finally, the zero algebra $0$ is non-degenerate evolution algebra since it admits an empty natural basis, and it is a unit for this operation.

\item The \emph{tensor product} $X\otimes Y$: 
Let $X$ and $Y$ be evolution algebras over a field $\K$, each endowed with natural bases $B=\{e_i\}_{i\in I}$ and $B'=\{f_j\}_{j\in J}$ respectively. It is well known, e.g.\ \cite[Theorem 3.2]{CabreraMartinMartinTocino2023},  that the tensor product algebra $X \otimes Y$ admits a natural basis given by $B\otimes B'=\{e_i \otimes f_j\}_{(i,j)\in I\times J}$, with multiplication determined by
\[
(e_i \otimes f_j)^2 = e_i^2 \otimes f_j^2, \quad (e_i \otimes f_j)(e_{i'} \otimes f_{j'}) = 0 \text{ for } (i,j)\neq(i',j').
\]
Hence $X\otimes Y$ is itself an evolution algebra, its structure matrix is the Kronecker product of the structure matrices of $X$ and $Y$ \cite[Remark 3.3]{CabreraMartinMartinTocino2023}. Moreover, if both $X$ and $Y$ are non-degenerate, then for every basis element $e_i \otimes f_j$ one has
\[
(e_i \otimes f_j)^2 = e_i^2 \otimes f_j^2 \neq 0,
\]
since both tensor factors are nonzero, and $X\otimes Y$ is again non-degenerate 
\cite[Lemma 3.8]{CabreraMartinMartinTocino2023}. Therefore, the class of non-degenerate evolution algebras is closed under tensor products. The one-dimensional algebra $\K$, acts as a unit for this operation.
\end{itemize}

Since the usual direct sum and tensor product of $\K$-algebras are associative and symmetric up to canonical
isomorphism, it follows that the category $\Evol_\K$ inherits a symmetric monoidal structure from the ambient
category of $\K$-algebras when considering these two operations. This is made explicit in the following statement.

\begin{proposition}\label{prop:monoidal}
$(\Evol_\K,\oplus,0)$ and $(\Evol_\K,\otimes,\K)$ are symmetric monoidal categories, and $\otimes$ is distributive with respect to $\oplus$:
\[
  X\otimes(Y\oplus Z)\;\iso\;(X\otimes Y)\oplus(X\otimes Z).
\]
\end{proposition}
\begin{proof}
The symmetric monoidal structures are inherited from the ambient category of
$\K$-algebras, as noted above; only the distributivity isomorphism requires
comment. Let $B_X$, $B_Y$, $B_Z$ be natural bases of $X$, $Y$, $Z$. On the
one hand, $Y\oplus Z$ has natural basis $B_Y\sqcup B_Z$, so $X\otimes(Y\oplus
Z)$ has natural basis
\[
  B_X\otimes(B_Y\sqcup B_Z)=(B_X\otimes B_Y)\sqcup(B_X\otimes B_Z).
\]
On the other hand, this is precisely the natural basis of $(X\otimes
Y)\oplus(X\otimes Z)$. The linear bijection matching $e\otimes f$ to itself
under this identification preserves squares, since
$(e\otimes f)^2=e^2\otimes f^2$ in both algebras and products of distinct
natural basis vectors vanish on both sides; hence it is an isomorphism of
evolution algebras.
\end{proof}
The distributivity in \autoref{prop:monoidal} is exactly the compatibility one needs
between an ``addition'' and a ``multiplication'' to form a ring out of
isomorphism classes, which is the subject of the next section.

\section{The Grothendieck ring of $\Evol_\K$}\label{sec:grothendieck}

In this section we introduce $\G(\Evol_\K)$, the Grothendieck ring of finite-dimensional non-degenerate evolution algebras and some of its basic properties.

Write $[X]$ for the isomorphism class of an evolution $\K$-algebra $X$
.
\begin{definition}\label{def:groth}
The \emph{Grothendieck ring} of $\Evol_\K$ is the ring
\[
  \G(\Evol_\K)=(G,+,\cdot)
\]
generated by the symbols $[X]$, for $X$ an object in $\Evol_\K$, subject to the relations
\begin{align}
  [X]=[Y] &\iff X\iso Y, \label{eq:iso}\\
  [X]+[Y] &= [X\oplus Y], \label{eq:add}\\
  [X]-[Y]=[Z] &\iff [X]=[Z\oplus Y], \label{eq:sub}\\
  [X]\cdot[Y] &= [X\otimes Y]. \label{eq:prod}
\end{align}
\end{definition}
Relations \eqref{eq:add}--\eqref{eq:sub} record the additive (group
completion) structure coming from $\oplus$, while \eqref{eq:prod}
introduces the multiplication coming from $\otimes$; that the latter is
well defined and distributes over the former is guaranteed by \autoref{prop:monoidal}.
This construction is well defined. By \cite[Theorem~5.11]{CabreraSilesVelasco2016} the commutative monoid $(\{[X]\},\oplus)$ of isomorphism classes is free on the indecomposables, hence cancellative, so its group completion is the free abelian group on the indecomposables and the natural map into it is injective. The product $\otimes$ respects isomorphism and distributes over $\oplus$ by \autoref{prop:monoidal}, so it descends to a well-defined, commutative, distributive multiplication on this group completion. In particular, collecting the positive and negative coefficients of a $\Z$-combination $\sum_k c_k[Y_k]$ as $[A]-[B]$, with $A=\bigoplus_{c_k>0}Y_k^{\oplus c_k}$ and $B=\bigoplus_{c_k<0}Y_k^{\oplus(-c_k)}$, produces genuine non-degenerate algebras, since $\Evol_\K$ is closed under $\oplus$ (\autoref{sec:monoidal}). This is what \autoref{lem:lovasz} and \autoref{thm:zero-divisors} use when reading off $X\otimes A\iso X\otimes B$.

\begin{remark}\label{rem:unity}
The Grothendieck ring $\G(\Evol_\K)$ is a commutative ring with unity. As it was mentioned in the previous section, the additive zero is the class of the trivial (zero) evolution algebra,
$$
  0=[0],\qquad [X]+0=[X\oplus 0]=[X],
$$
while multiplicative unit is the class of the one-dimensional
evolution algebra $\K=\K\{e\}$, with $e^2=e$, taken with the natural
basis $\{e\}$,
\[
  1=[\K],\qquad [X]\cdot 1=[X]\cdot[\K]=[X\otimes\K]=[X],
\]
the latter because $\K$ is the unit for $\otimes$
(\autoref{prop:monoidal}).
\end{remark}

As an abelian group, $\G(\Evol_\K)$ is free. Indeed, since every finite-dimensional non-degenerate evolution algebra decomposes in a unique way as a direct sum of indecomposable evolution algebras \cite[Theorem 5.11]{CabreraSilesVelasco2016}, so the isomorphism classes of finite-dimensional indecomposables non-degenerate evolution algebras form a $\Z$-basis:
\[
  (\G(\Evol_\K),+)\;=\;\bigoplus_{[X]\ \text{indecomp.}} \Z\,[X].
\]
This group is \emph{huge}, and listing a set of additive generators is literally the wild classification problem from the introduction.

As a \emph{ring}, by contrast, $\G(\Evol_\K)$ is generated by far fewer elements, because the product \eqref{eq:prod} imposes relations, in particular many zero-divisors, among the additive generators. Concretely, two phenomena reduce the number of generators one must list.
\medskip

\noindent\textbf{(a) Indecomposable products are not needed as ring generators.}
There exist indecomposable non-degenerate evolution algebras $X$ and $Y$ such that the tensor product $X\otimes Y$ is again indecomposable (see \autoref{thm:indecomp-tensor}). In this case,
the three classes $[X]$, $[Y]$ and $[X\otimes Y]$ are distinct additive basis
elements and would all have to be recorded. However, as ring elements,
$[X\otimes Y]=[X][Y]$ by \eqref{eq:prod}, so once $[X]$ and $[Y]$ are
known, the class $[X\otimes Y]$ is determined, and need not be listed separately as a generator.
\medskip

\noindent\textbf{(b) Non-cancellation phenomena produce zero-divisors.} As we show in the following pages, there exist non-degenerate evolution algebras $X$, $Y$, and $Z$ such that
$$
  X\otimes Y \iso X\otimes Z \qquad\text{even though}\qquad Y\not\iso Z .
$$
In the Grothendieck ring $\G(\Evol_\K)$, this relation becomes
$$
  [X]\bigl([Y]-[Z]\bigr)=0
  \qquad\text{with}\qquad
  [Y]-[Z]\ne 0,
$$
which proves that $[X]$ is a zero-divisor. Consequently, the presence of these relations implies that $\G(\Evol_\K)$ is not a free ring. This imposes constraints on the generators: we do not need to record all products individually, as many are already determined by these algebraic dependencies.

To conclude this section, we illustrate how algebraic manipulations in $\G(\Evol_\K)$ give rise to identities that may be of independent interest, as they exhibit further instances of non-cancellation phenomena.

\begin{lemma}\label{lem:lovasz}
Let $A_i$, $B_i$ and $D_i$, $i=1,2,$ be non-degenerate evolution algebras. Suppose $A_i\otimes D_i\iso B_i\otimes D_i$ for $i=1,2$. Then
\[
  (A_1\otimes A_2\;\oplus\;B_1\otimes B_2)\otimes(D_1\oplus D_2)
  \;\iso\;
  (A_1\otimes B_2\;\oplus\;B_1\otimes A_2)\otimes(D_1\oplus D_2).
\]
\end{lemma}
\begin{proof}
In $\G(\Evol_\K)$ the hypotheses $A_i\otimes D_i\iso B_i\otimes D_i$ read
\[
  ([A_i]-[B_i])\,[D_i]=0 \qquad (i=1,2).
\]
Therefore
\[
  P:=([A_1]-[B_1])\,([A_2]-[B_2])\,\bigl([D_1]+[D_2]\bigr).
\]
On the one hand, multiplying out and using \eqref{eq:add}--\eqref{eq:prod},
\[
\begin{split}
  P
  &= \bigl[(A_1\otimes A_2)\oplus(B_1\otimes B_2)\bigr]\bigl[D_1\oplus D_2\bigr]\\
  &\quad - \bigl[(A_1\otimes B_2)\oplus(B_1\otimes A_2)\bigr]\bigl[D_1\oplus D_2\bigr],
\end{split}
\]
which is exactly the difference of the classes of the two sides of the claimed
isomorphism. On the other hand, expanding $P$ over $[D_1]+[D_2]$ and using
commutativity of $\G(\Evol_\K)$,
\[
  P=([A_2]-[B_2])\,\underbrace{([A_1]-[B_1])[D_1]}_{=0}
  +([A_1]-[B_1])\,\underbrace{([A_2]-[B_2])[D_2]}_{=0}=0 ,
\]
each summand vanishing by the hypotheses. Hence the difference of the two
classes is $0$, that is, they are equal; since $\oplus$-decomposition into indecomposables is essentially unique, equal classes come from isomorphic algebras, which is the stated isomorphism.
\end{proof}

\section{Balance of non-degenerate algebras is an invariant}\label{sec:balance}

In the following sections we show that whether or not an indecomposable object $X$ in $\Evol_\K$ gives rise to a zero divisor in $\G(\Evol_\K)$ is controlled by the balance of $X$. But a non-degenerate evolution algebra $X$ can carry genuinely different natural bases $B$ with different associated graphs (\autoref{rem:basisnotunique}), thus the quantity $\b(\Gamma(X,B))$
is a priori basis-dependent. In this section we show that is not the case.

\begin{theorem}\label{thm:balance_invariant}
Let $A\neq 0$ be a finite-dimensional non-degenerate evolution $\K$-algebra. Then $\b(\Gamma(A,B))$ does not depend on the chosen natural basis $B$ of $A$.
\end{theorem}
\begin{proof}
Since $A$ is non-degenerate, then $b^2\ne 0$ for all $b\in B,$ and every vertex of \(\Gamma(A,B)\) has at least one outgoing
arrow; equivalently, \(\Gamma(A,B)\) has no sinks.

Since the graph $\Gamma(A,B)$ is finite, starting at any vertex and successively following an outgoing arrow eventually produces a repeated vertex. Thus $\Gamma(A,B)$ contains a directed cycle, and in particular,
$\b(\Gamma(A,B))$ is a well-defined positive integer.

We now use the canonical evolution rank-one decomposition of $A$: by
\cite[Theorem~2.11]{BoudiCabreraSiles2022}, there is a
decomposition
\begin{equation}\label{eq:evol_rank_decomp}
A=E_1\oplus\cdots\oplus E_r
\end{equation}
such that
$$\dim E_i^2=1,
\qquad\text{and}\qquad 
E_iE_j=0,\quad
\dim(E_i^2+E_j^2)=2
\quad
(i\neq j).$$

Because $A$ is non-degenerate,  the decomposition given in \autoref{eq:evol_rank_decomp}
is unique, up to a permutation of its summands. Moreover, the proof of
\cite[Theorem~2.11, p.~165]{BoudiCabreraSiles2022} also shows that any natural basis $B$ of $A$ admits a partition
\begin{equation}\label{eeq:partition_base}
B=B_1\mathbin{\dot\cup}\cdots\mathbin{\dot\cup}B_r
\end{equation}
such that $B_i$ is a basis of $E_i$. 

We are defining a graph $\overline{\Gamma}(A)$ associated to the decomposition given in \autoref{eq:evol_rank_decomp}. For each $i$, choose a nonzero vector
$q_i\in E_i^2.$ Since \autoref{eq:evol_rank_decomp} holds, there are uniquely determined vectors $q_{ij}\in E_j$ such that
$$q_i=q_{i1}+\cdots+q_{ir}.$$
Then $\overline{\Gamma}(A)$ is the graph whose vertices are
$E_1,\dots,E_r,$
and whose edges are determined by
$$(E_i, E_j)\in \overline{\Gamma}(A)
\quad\Longleftrightarrow\quad
q_{ij}\neq 0.$$

This definition of $\overline{\Gamma}(A)$ is independent of the choice of the nonzero generator
$q_i\in E_i^2$, because replacing $q_i$ by a nonzero scalar multiple (recall that $\dim E_i^2=1$) does
not change which components $q_{ij}$ vanish. The uniqueness clause
in \cite[Theorem~2.11]{BoudiCabreraSiles2022} therefore implies that
$\overline{\Gamma}(A)$ is intrinsic to $A$, up to relabelling its vertices.

We now prove that $\b(\Gamma(A,B))
=
\b(\overline{\Gamma}(A))$
by showing that $\overline{\Gamma}(A)$ is the quotient graph of $\Gamma(A,B)$ induced by the natural projection on vertices.

Let $B=\mathbin{\dot\cup}_{i=1}^r B_i$ be the partition provided by \autoref{eeq:partition_base}. Since $A$ is non-degenerate, for each $e\in B_i$ there exists $\lambda_{e}\in \K^\times$ such that
$$e^2=\lambda_{e} q_i=\lambda_{e}(q_{i1}+\cdots+q_{ir}).$$

For every pair $i,j$, define
$$S_{ij}(B)
=
\left\{
f\in B_j:
\text{the coefficient of \(f\) in \(q_{ij}\) is nonzero}
\right\}.$$
Then
$S_{ij}(B)\neq\varnothing$ if and only if 
$(E_i, E_j)$ is an edge in $\overline{\Gamma}(A).$
Moreover, for any $e\in B_i$ and $f\in B_j$, $(e, f)$ is an edge in $\Gamma(A,B)$ if and only if $f\in S_{ij}(B).$
Thus all vertices belonging to the same block $B_i$ have the same
successors inside every block $B_j$.
In other words, the map given at level of vertices by 
$$
\begin{array}{rcl}
\pi_B\colon \Gamma(A,B) & \longrightarrow & \overline{\Gamma}(A)\\
e & \longmapsto & E_i
\quad\text{for }e\in B_i.
\end{array}
$$
projects every edge in $\Gamma(A,B)$, with its orientation preserved, to an edge in $\overline{\Gamma}(A)$. 

We first observe that any two vertices in the same block $B_j$
can be connected by a path of balance zero. Let $x,y\in B_i$.
Because $q_i\neq 0$, there exist an index $j$ and a vertex
$z\in S_{ij}(B)$. Hence
$(x, z)$ and $(y,z)$ are edges in $\Gamma(A,B)$, and the path
$$\gamma_{x,y}:=(x, (x, z), z, (y,z), y)  $$
connects $x$ to $y$ and has balance $\b(\gamma_{x,y})=1-1=0$.

Then, for each elementary path of balance $1$ in $\overline{\Gamma}(A)$, namely $\overline{\epsilon}=(E_i,(E_i,E_j),E_j)$, choose a fixed elementary path of balance $1$ in $\Gamma(A,B)$, say ${\epsilon}=(x,(x,y),y)$, where $x\in B_i$ and $y\in S_{ij}(B).$ Therefore $\pi_B({\epsilon})=\overline{\epsilon},$ and we say that ${\epsilon}$ is the preferred lift of $\overline{\epsilon}$.

Now let $\gamma$ be a cycle in $\Gamma(A,B)$. Its projection
$\pi_B(\gamma)$ is a cycle in $\overline{\Gamma}(A)$, and the
orientation of every traversed edge is preserved. Therefore
\[
\b(\pi_B(\gamma))=\b(\gamma).
\]
It follows that every cycle balance occurring in \(\Gamma(A,B)\)
also occurs in \(\overline{\Gamma}(A)\). Consequently,
\begin{equation}\label{eq:balance_1}
\b(\overline{\Gamma}(A))
\mid
\b(\Gamma(A,B)).
\end{equation}

Conversely, let $\overline{\gamma}$ be a cycle in $\overline{\Gamma}(A)$. Write $\overline{\gamma}=\operatornamewithlimits{\star}_{k=1}^r \overline{\epsilon}_k^{a_k},$ a concatenation
where every $\overline{\epsilon}_k$ is an elementary path of balance $1$ and $a_k\in\{-1,1\},$ thus $\b(\overline{\gamma})=\sum_{k=1}^r a_k.$ Let ${\epsilon}_k$ be the preferred lift of $\overline{\epsilon}_k$, and let $s_k$ and $t_k$ be the initial and final vertices of the path ${\epsilon}_k^{a_k}$. Since $\overline{\epsilon}_k^{a_k}=\pi_B({\epsilon}_k^{a_k})$ and $\overline{\epsilon}_{k+1}^{a_{k+1}}=\pi_B({\epsilon}_{k+1}^{a_{k+1}})$ are concatenable paths in $\overline{\Gamma}(A)$ (indices ${\mod r})$, then $t_k,s_{k+1}\in B_j\subset B$ for some $j$, and there exists a balance $0$ path $\gamma_{t_k,s_{k+1}}$ connecting these points in $\Gamma(A,B)$. Therefore $\gamma=\operatornamewithlimits{\star}_{k=1}^r ({\epsilon}_k^{a_k}\star\gamma_{t_k,s_{k+1}})$ is a cycle in $\Gamma(A,B)$ (recall that $s_{k+1}=s_0$) with
$$\b(\gamma)=\sum_{k=1}^r \big(\b({\epsilon}_k^{a_k})+\b(\gamma_{t_k,s_{k+1}})\big)=\sum_{k=1}^r ({a_k}+0)=\b(\overline{\gamma}).$$
Therefore every cycle balance occurring in
$\overline{\Gamma}(A)$ also occurs in $\Gamma(A,B)$, so
\begin{equation}\label{eq:balance_2}
\b(\Gamma(A,B))
\mid
\b(\overline{\Gamma}(A)).
\end{equation}

Combining \autoref{eq:balance_1} and \autoref{eq:balance_2}, we obtain
\[
\b(\Gamma(A,B))
=
\b(\overline{\Gamma}(A)).
\]
The graph \(\overline{\Gamma}(A)\) is intrinsic to $A$, so the
right-hand side does not depend on $B$. This proves the result.
\end{proof}

\begin{remark}\label{rm:definition_balance_not_B}
Since for any given finite-dimensional non-degenerate evolution $\K$-algebra $A\neq 0$, the value $\b(\Gamma(A,B))$ does not depend on the chosen natural basis $B$ of $A$, we shall simply write $\b(A)$. In particular $\b(A)$ is well defined for every non-degenerate $A$.
\end{remark}

\section{The rigidified category $\NEvol_\K$}\label{sec:nevol}

As it was mentioned in \autoref{sec:prelim}, the main difficulty when working with evolution algebras is that algebra homomorphisms ignore the natural basis. We therefore work in a category in which the natural basis, or rather the resulting decomposition into
lines, is part of the structure that morphisms must preserve.

\begin{definition}
The \emph{category of natural evolution algebras} over $\K$ is the category whose objects and morphisms are:

\begin{description}
    \item[Objects] A \emph{natural evolution algebra} is a non-degenerate evolution algebra $X$, together with a fixed decomposition 
\[
  X=\bigoplus_{i=1}^{n} A_i,
  \qquad A_iA_j=0 \ (i\neq j),
  \qquad \dim_\K A_i = 1,
\]
into one-dimensional subspaces with pairwise trivial products.

\item[Morphisms] A morphism in $\NEvol_\K$ is a $\K$-algebra homomorphism that preserves
the decomposition, that is, an algebra morphism $f\colon \bigoplus_i A_i \to \bigoplus_j B_j$ such that
each line is sent into a line,
\[
  f(A_i)\subseteq B_j \quad\text{for some } j .
\]
Equivalently, $f$ comes with a map of index sets $\hat f\colon I\to J$
with
\[
  f(A_i)\subseteq B_{\hat f(i)} \qquad (i\in I). 
\]
We further require that no line is annihilated: $f(A_i)\neq 0$ for every $i$ (equivalently, $f$ restricts to a linear isomorphism $A_i\xrightarrow{\iso}B_{\hat f(i)}$), so that the index map $\hat f$ is well defined.
\end{description}
\end{definition}

\begin{remark}\label{rem:twobases}

Observe that there is a forgetful functor
\[
  \NEvol_\K \rightsquigarrow \Evol_\K .
\]
given by forgetting the natural decomposition. This is a faithful functor which is clearly not an embedding: A single evolution algebra $X$ may carry two different natural bases $B_1,B_2$, giving two different objects
\[
  X_1=\bigoplus_{b\in B_1}\K\{b\},
  \qquad
  X_2=\bigoplus_{b\in B_2}\K\{b\}
\]
in $\NEvol_\K$ that become equal, $X_1=X_2$, after applying the forgetful
functor. 
\end{remark} 

\begin{remark}\label{rm:nat_isomorphisms}
An isomorphism $\varphi\colon X\xrightarrow{\iso}Y$ of evolution
algebras carrying $B_1$ to a natural basis $B_2=\varphi(B_1)$ of $Y$ is
the prototype of the data the rigidified category $\NEvol_\K$ keeps track of. Indeed, a morphism in $\NEvol_\K$ permutes the lines and rescales them, and therefore
the automorphisms of an object are monomial: writing $n=\dim X$,
\[
  \Aut_{\NEvol_\K}(X)\;\le\;\K^{\times}\wr S_n ,
\]
the group of generalized permutation ($\K$-monomial) matrices. Explicitly
an automorphism is a pair
\[
  \varphi=(\sigma;\lambda_1,\dots,\lambda_n),
  \qquad \sigma\in S_n,\ \lambda_i\in\K^{\times},
  \qquad \varphi(A_i)=\lambda_i\,A_{\sigma(i)} .
\] 
The cancellation phenomena described in \autoref{sec:grothendieck} are governed
by a multi-indexed refinement of an automorphism.

\end{remark}

Every natural evolution algebra has a unique well-defined associated graph.  

\begin{definition}\label{def:natural_grpah}
Let $X=\bigoplus_{i=1}^{n} A_i$ be a natural evolution algebra over $\K$.
Its \emph{associated graph} is the directed graph
\[
  \Gamma(X)=(V,E),\qquad V=\{A_1,\dots,A_n\},
\]
whose vertices are the one-dimensional  $A_i$, and where
$(A_r,A_s)\in E$ if and only if the composite
\[
  A_r^{2}\;\hookrightarrow\; X=\bigoplus_{i=1}^{n} A_i
  \;\xrightarrow{\ \pi_s\ }\; A_s
\]
is nonzero, where $\pi_s$ is the canonical linear projection
$$\begin{array}{rcl}
 X=\bigoplus_{i=1}^{n} A_i & \buildrel \pi_s \over \longrightarrow & A_s,\\
  \textstyle\sum_i x_i & \longmapsto & x_s\quad (x_i\in A_i).
\end{array}$$
For a natural evolution algebra $X$ we set $\b(X):=\b\bigl(\Gamma(X)\bigr)$.
\end{definition}

\begin{remark}\label{rm:definition_balance_not_B_II}
If $A=\bigoplus_{i=1}^{n} b_i\K$ is a non-degenerate natural evolution algebra over $\K$, then $B=\{b_i\}$ is a natural base of $A$ and the graphs $\Gamma(A,B)$ and $\Gamma(A)$ coincide. Since the balance of non-degenerate evolution algebra is independent of the chosen natural basis (\autoref{thm:balance_invariant}), the different notions of balance $\b(A)$ agree. 
\end{remark}

We now introduce a construction that will play a central role in \autoref{sec:cyclic}, where it will be used to describe non-degenerate evolution algebras $X,Y,Z$ satisfying
$$X\otimes Y \cong X\otimes Z$$
while $Y\not\cong Z$.

\begin{definition}\label{def:miso}
Let $X=\bigoplus_{i=1}^n \K b_i$ be an object in $\NEvol_\K$ with structure matrix $(m_{kl})$ associated to the basis $B=\{b_i\}$. An
\emph{$m$-isomorphism} of $X$ is an
$m$-tuple
\[
  \Phi=(\varphi_0,\dots,\varphi_{m-1}),
  \qquad
  \varphi_r=(\sigma_r;\lambda_{r,1},\dots,\lambda_{r,n})\in
 \K^{\times}\wr S_n,
\]
such that the scalars
\[
  \mu_{ij}\;:=\;
  \frac{\,m_{\,\sigma_r^{-1}(i)\,,\,\sigma_{r+1}^{-1}(j)}\;\cdot\;
        \lambda_{\,r+1,\,\sigma_{r+1}^{-1}(j)}\,}
       {\lambda_{\,r,\,\sigma_r^{-1}(i)}^{\,2}}
\]
do not depend on $r$ (indices read modulo $m$). The scalars $(\mu_{ij})$
are then the structure constants of an evolution algebra, which we denote
by $X^{\Phi}$:
\[
  X^{\Phi}\;:=\;\bigoplus_{i=1}^{n} \K \bar b_i
  \qquad\text{with structure matrix } (\mu_{ij}).
\]
\end{definition}

We first show that $m$-isomorphisms preserve non-degeneracy for evolution algebras.

\begin{proposition}
Let $X=\bigoplus_{i=1}^n \K b_i$ be an object in $\NEvol_\K$ with structure matrix $(m_{kl})$ associated to the basis $B=\{b_i\}$. 
Let $\Phi=(\varphi_0,\dots,\varphi_{m-1})$ be an $m$-isomorphism of $X$. Then the evolution
algebra $X^{\Phi}=\bigoplus_{i=1}^{n} \K \bar b_i$ is non-degenerate, that is, it is an object in $\NEvol_\K$.
\end{proposition}
\begin{proof}
Recall that an evolution algebra with natural basis $\{b_1,\ldots,b_n\}$ is non-degenerate precisely when every natural basis element has nonzero square.
Equivalently, every row of its structure matrix is nonzero.

We prove that every row of the structure matrix $(\mu_{ij})$ of $X^{\Phi}$ is nonzero. Fix $i\in\{1,\ldots,n\}$. Choose any $r\in\{1,\ldots,m\}$. By definition,
\[
\mu_{ij}
=
\frac{
 m_{\sigma_r^{-1}(i),\sigma_{r+1}^{-1}(j)}
 \lambda_{r+1,\sigma_{r+1}^{-1}(j)}
}{
 \lambda_{r,\sigma_r^{-1}(i)}^2
}.
\]

Since each $\varphi_r$ is an automorphism and $\varphi_r(b_k)=\lambda_{r,k}b_{\sigma_r(k)},$
we have $\lambda_{r,k}\neq 0$ for every $r$ and every $k$. Hence 
\begin{equation}\label{eq:prop_m_auto_1}
 \frac{\lambda_{r+1,\sigma_{r+1}^{-1}(j)}}
     {\lambda_{r,\sigma_r^{-1}(i)}^2}\ne 0   
\end{equation}
for every $i,j$.

Now put $k=\sigma_r^{-1}(i).$ Since $X$ is non-degenerate, the $k$-th row of $(m_{ij})$ is nonzero. Therefore
there exists some $\ell\in\{1,\ldots,n\}$ such that
\begin{equation}\label{eq:prop_m_auto_2}
    m_{k\ell}\neq 0.
\end{equation}
Take $j=\sigma_{r+1}(\ell),$ or equivalently $\sigma_{r+1}^{-1}(j)=\ell.$
Therefore, combining \autoref{eq:prop_m_auto_1} and \autoref{eq:prop_m_auto_2},
\[
\mu_{ij}
=
\frac{
 m_{\sigma_r^{-1}(i),\sigma_{r+1}^{-1}(j)}
 \lambda_{r+1,\sigma_{r+1}^{-1}(j)}
}{
 \lambda_{r,\sigma_r^{-1}(i)}^2
}
=
\frac{m_{k\ell}\lambda_{r+1,\ell}}{\lambda_{r,k}^2}\neq 0.
\]
Thus the $i$-th row of $(\mu_{ij})$ is nonzero.

Since $i$ was arbitrary, every row of the structure matrix $(\mu_{ij})$ is nonzero. Hence every natural basis element of $X^{\Phi}$ has nonzero square.
Therefore $X^{\Phi}$ is non-degenerate.
\end{proof}

\begin{example}[The cyclic algebra $C_n$]\label{ex:miso-cyclic}
Let $X=C_n=\K c_0\oplus\dots\oplus\K c_{n-1}$ be the cyclic evolution
algebra with $c_i^{\,2}=c_{i+1}$ (indices mod $n$), so that its structure
matrix is
\[
  m_{ij}=
  \begin{cases}
    1, & j= i+1 \pmod n,\\
    0, & \text{otherwise.}
  \end{cases}
\]
For $r\in\Z/n$ let $\sigma_r\in S_n$ be the cyclic shift
$\sigma_r(i)=i-r \pmod n$ and set
$$\varphi_r=(\sigma_r;1,\dots,1)\in\Aut_{\NEvol_\K}(C_n)\leq \K^{\times}\wr S_n.$$ Then
$\Phi=(\varphi_0,\dots,\varphi_{n-1})$ is an $n$-isomorphism of $C_n$:
since all scalars equal $1$,
\[
  \mu_{ij}
   = m_{\sigma_r^{-1}(i),\,\sigma_{r+1}^{-1}(j)}
   = m_{\,i+r,\;j+r+1}
   = \begin{cases} 1, & i= j \pmod n,\\ 0,&\text{otherwise,}\end{cases}
\]
which is independent of $r$. The corresponding algebra $C_n^{\Phi}$ is therefore
\[
  C_n^{\Phi}=\bigoplus_{j=0}^{n-1}\K\{\bar c_j\},
  \qquad \bar c_j^{\,2}=\bar c_j ,
\]
the split algebra $\K^{\,n}$. In particular $C_n^{\Phi}\not\iso C_n$ for
$n\ge 2$.
\end{example}

Thus an $m$-isomorphism of a non-degenerate evolution algebra $X$ produces a new non degenerate evolution algebra $X^\Phi$, in general not isomorphic to $X$.

The following result shows that two distinct $m$-isomorphisms of a non-degenerate evolution algebra $X$, say $\Phi$ and $\Phi'$, may give rise to isomorphic evolution algebras $X^\Phi \cong X^{\Phi'}$.

\begin{lemma}\label{lem:cyclic-shift}
Let $X=\bigoplus_{i=1}^n \K b_i$ be an object in $\NEvol_\K$ with structure matrix $(m_{kl})$ associated to the basis $B=\{b_i\}$. 
Let $\Phi=(\varphi_0,\dots,\varphi_{m-1})$ be an $m$-isomorphism of $X$. Then
every cyclic ``shift" of $\Phi$,
$\Phi^{(s)}=(\varphi_s,\varphi_{s+1},\dots,\varphi_{s+m-1})$ (indices mod
$m$) is an $m$-isomorphism of $X$, and $X^\Phi \cong X^{\Phi^{(s)}}$.
\end{lemma}
\begin{proof}
The condition of \autoref{def:miso} relates only consecutive pairs
$\varphi_r,\varphi_{r+1}$, and therefore any cyclic ``shift" preserves every such pair. Moreover, by \autoref{def:miso} the structure constants $\mu_{ij}$ of $X^{\Phi}$ are independent of $r$; those of $X^{\Phi^{(s)}}$ are given by the same expression evaluated at $r\mapsto r+s$, hence take the identical value $\mu_{ij}$. Thus $X^{\Phi}$ and $X^{\Phi^{(s)}}$ carry the same structure matrix on the same basis, so $X^{\Phi}=X^{\Phi^{(s)}}$; in particular $X^{\Phi}\cong X^{\Phi^{(s)}}$.
\end{proof}

\section{Cyclic evolution algebras and cancellation I}\label{sec:cyclic}

In this section we obtain our first results on the zero divisors of $\G(\Evol_\K)$, or equivalently, on the non-cancellation phenomena in $\Evol_\K$ or  $\NEvol_\K$. First we need to introduce some notation.

\begin{definition}\label{def:relC}
Let $A,B,C$ be evolution $\K$-algebras. We say that $A\otimes C$ and
$B\otimes C$ are \emph{isomorphic relative to $C$} if there is an
isomorphism
\[
  \varphi\colon A\otimes C \xrightarrow{\ \iso\ } B\otimes C
\]
that preserves the second tensor coordinate: for every $c\in C$ there is a
linear map $\psi_c\colon A\to B$ with
\[
  \varphi(a\otimes c)=\psi_c(a)\otimes c \qquad(a\in A).
\]
Equivalently, $\varphi(A\otimes c)\subseteq B\otimes c$ for every $c\in C$. In this case we say $\varphi$ is \emph{second-coordinate preserving}.
\end{definition}

The relevance of the cyclic algebras, defined in \autoref{ex:miso-cyclic}, to the multiplicative structure of
$\G(\Evol_\K)$ is that tensoring with $C_m$ cancels exactly up to the
$m$-isomorphisms of \autoref{def:miso}.

\begin{theorem}\label{thm:cancel-Cn}
Let $A,B$ be algebras in $\NEvol_\K$ of dimension $n$. Then
\[
  A\otimes C_m \;\iso\; B\otimes C_m \qquad(\text{relative to } C_m)
\]
in $\NEvol_\K$ if and only if there is an $m$-isomorphism
$\Phi=(\varphi_0,\dots,\varphi_{m-1})$ of $A$ with $B\iso A^{\Phi}$.
\end{theorem}

\begin{proof}
Let $A=\bigoplus_{i=1}^n \K a_i$ and $B=\bigoplus_{i=1}^n \K b_i$, so $\{a_i\}_{i=1}^{n}$ and $\{b_i\}_{i=1}^{n}$ are natural bases of $A$
and $B$, with structure constants $a_i^2=\sum_j m_{ij}a_j$ and
$b_i^2=\sum_j \mathcal{M}_{ij}b_j$. Following the notation in \autoref{ex:miso-cyclic}, let $C_m=\bigoplus_{i=0}^{m-1} \K c_i$, so $c_r^2=c_{r+1}$ (indices mod $m$). An
isomorphism 
\[
  \varphi\colon A\otimes C_m \;\iso\; B\otimes C_m \qquad(\text{relative to } C_m)
\]
in $\NEvol_\K$ has the form
\[
  \varphi(a_i\otimes c_r)=\varphi_r(a_i)\otimes c_r
   =\lambda_{ri}\,b_{\sigma_r(i)}\otimes c_r,
   \qquad \varphi_r=(\sigma_r;\lambda_{r1},\dots,\lambda_{rn}),
\]
with one monomial map $\varphi_r$ per layer $c_r$. We determine when
$\varphi$ is an algebra isomorphism.

\emph{Orthogonality.} For $i\neq j$,
\[
  \varphi(a_i\otimes c_r)\,\varphi(a_j\otimes c_r)
   =\lambda_{ri}\lambda_{rj}\,(b_{\sigma_r(i)}b_{\sigma_r(j)})\otimes c_r^2,
\]
which vanishes because $\sigma_r\in S_n$ is injective, so
$\sigma_r(i)\neq\sigma_r(j)$ and $b_{\sigma_r(i)}b_{\sigma_r(j)}=0$. This
holds automatically.

\emph{Squares.} On the one hand,
\[
  \varphi\bigl((a_i\otimes c_r)^2\bigr)
   =\varphi\Bigl(\textstyle\sum_j m_{ij}\,a_j\otimes c_{r+1}\Bigr)
   =\sum_j m_{ij}\,\lambda_{r+1,j}\,b_{\sigma_{r+1}(j)}\otimes c_{r+1},
\]
while on the other,
\[
  \varphi(a_i\otimes c_r)^2
   =(\lambda_{ri}b_{\sigma_r(i)}\otimes c_r)^2
   =\lambda_{ri}^2\sum_j \mathcal{M}_{\sigma_r(i),\,\sigma_{r+1}(j)}\,
     b_{\sigma_{r+1}(j)}\otimes c_{r+1}.
\]
Equality for all $i,j,r$ amounts to
$m_{ij}\lambda_{r+1,j}=\lambda_{ri}^2\,\mathcal{M}_{\sigma_r(i),\sigma_{r+1}(j)}$;
substituting $i=\sigma_r^{-1}(a)$ and $j=\sigma_{r+1}^{-1}(b)$ gives
\[
  \mathcal{M}_{ab}
   =\frac{m_{\sigma_r^{-1}(a),\,\sigma_{r+1}^{-1}(b)}\;
          \lambda_{r+1,\,\sigma_{r+1}^{-1}(b)}}
         {\lambda_{r,\,\sigma_r^{-1}(a)}^{\,2}}
   \qquad(\text{for all }\,a,b,r).
\]
That the right-hand side is independent of $r$ is precisely the condition
that $\Phi=(\varphi_0,\dots,\varphi_{m-1})$ be an $m$-isomorphism of $A$ as given in \autoref{def:miso}, and 
$\mathcal{M}_{ab}=\mu_{ab}$ says exactly that $B\iso A^{\Phi}$.
\end{proof}

\begin{example}[Non-cancellation of cyclic algebras]\label{ex:non-cancel--cyclic}
In \autoref{ex:miso-cyclic} it is shown that the cyclic evolution algebra $C_n$ admits an $n$-isomorphism $\Phi$ such that $C_n^\Phi=\K^n\not\cong C_n$. Therefore, according to \autoref{thm:cancel-Cn}, $C_n\otimes C_n\cong \K^n\otimes C_n$ while $\K^n\not\cong C_n$, and $[C_n]$ is a zero divisor in $\G(\Evol_\K)$.
\end{example}

\section{Cyclic evolution algebras and cancellation II}\label{sec:cyclic_II}

We now generalize the results in the previous section to a more general family of cyclic evolution algebras,
allowing also a ``shift" of the second coordinate when making tensor with these algebras. This generalization makes sense only when the base field $\K$ is not algebraically closed (see \autoref{lem:weighted-trivial} and \autoref{lem:shifted-trivial}).

\begin{definition}\label{def:cyclic}
Given scalars $m_0,\dots,m_{n-1}\in\K^{\times}$, the \emph{cyclic evolution
algebra} $C(m_0,\dots,m_{n-1})$ is defined as
$$
C(m_0,\dots,m_{n-1}):=\K c_0\oplus\dots\oplus\K c_{n-1},
  \qquad c_i^{\,2}=m_i\,c_{i+1}\quad(\text{indices mod }n). 
$$
The cyclic evolution algebra with all weights equal to $1$ is $C_n=C(1,\dots,1).$
\end{definition}

\begin{remark}\label{rem:cyclic-normal}
Observe that \[
  C(m_0,\dots,m_{n-1})\;\iso\;
  C\!\left(1,\dots,1,\ \textstyle\prod_{i=0}^{n-1}m_i^{2^{\,n-1-i}}\right).
\]
Indeed, if we denote by $\{a_0,\ldots,a_{n-1}\}$ the natural basis of $C(m_0,\dots,m_{n-1})$ and by $\{b_0,\ldots,b_{n-1}\}$ the natural basis of $C\!\left(1,\dots,1,\prod_{i=0}^{n-1}m_i^{2^{\,n-1-i}}\right)\!,$ the morphism given by $\psi(a_{0})=b_0$ and  $$\psi(a_{i+1})=\big(\prod_{j=0}^{i}m_j^{-2^{i-j}}\big)b_{i+1}\quad\text{for $0\le i\le n-2$,}$$
is the desired isomorphism. Observe that the isomorphism $\psi$ above is just a
rescaling the natural bases, and therefore, it defines an isomorphism in $\NEvol_\K$. \autoref{lem:weighted-trivial} below shows that, over an
algebraically closed field, even this scalar normalizes to $1$, so that
$C(m_0,\dots,m_{n-1})\iso C_n$ outright and, in particular, $C_n$ is the normal
form of every cyclic evolution algebra.
\end{remark}

We do now generalize the concept of $m$-isomorphism.

\begin{definition}\label{def:k-alpha-iso}
Let $\alpha_0,\dots,\alpha_{m-1}\in\K^{\times}$ and $k\in\Z/m$. A
\emph{$(k;\alpha_0,\dots,\alpha_{m-1})$-isomorphism} of an evolution
algebra $A$ with structure constants $(m_{ij})$ is a sequence
$\Phi=(\varphi_0,\dots,\varphi_{m-1})$,
$\varphi_r=(\sigma_r;\lambda_{r1},\dots,\lambda_{rn})$, such that the
scalars
\[
  \mu_{ij}\;:=\;
  \frac{m_{\sigma_r^{-1}(i),\,\sigma_{r+1}^{-1}(j)}\;
        \lambda_{r+1,\,\sigma_{r+1}^{-1}(j)}}
       {\lambda_{r,\,\sigma_r^{-1}(i)}^{\,2}}
  \cdot\frac{\alpha_r}{\alpha_{r+k}}
\]
are independent of $r$ (indices mod $m$). As before, $A^{\Phi}$ denotes
the evolution algebra on a copy of the basis of $A$ with structure
constants $(\mu_{ij})$.
\end{definition}

\begin{remark}\label{rem:k-alpha-special}
Taking all $\alpha_r$ equal, or more generally $\alpha_r/\alpha_{r+k}= 1$,
recovers the $m$-isomorphism of \autoref{def:miso}; in particular
\autoref{ex:miso-cyclic} is an instance of the
present notion as well.
\end{remark}

\begin{definition}\label{def:relC_shift}
Let $A,B$ be evolution $\K$-algebras and let
$C=C(\alpha_0,\dots,\alpha_{m-1})$. We say that $A\otimes C$ and
$B\otimes C$ are \emph{isomorphic relative to $C$ shifting the second coordinate by $k$} if there is an
isomorphism
\[
  \varphi\colon A\otimes C \xrightarrow{\ \iso\ } B\otimes C
\]
such that for every $c_r\in C$ there is a
linear map $\psi_r\colon A\to B$ with
\[
  \varphi(a\otimes c_r)=\psi_r(a)\otimes c_{r+k} \qquad(a\in A).
\]
\end{definition}

\begin{remark}\label{rem:shift}
An isomorphism relative to $C(\alpha_0,\dots,\alpha_{m-1})$ shifting the second coordinate by $0$ is just an isomorphism relative to $C(\alpha_0,\dots,\alpha_{m-1})$ as given in \autoref{def:relC}.
\end{remark}

\begin{theorem}\label{thm:cancel-weighted}
Let $A,B$ be algebras in $\NEvol_\K$ and let
$C=C(\alpha_0,\dots,\alpha_{m-1})$. Then $A\otimes C\iso B\otimes C$ relative to $C$  shifting the second coordinate by $k$ in $\NEvol_\K$ if and only if there is a $(k;\alpha_0,\dots,\alpha_{m-1})$-isomorphism $\Phi$ of $A$ with
$B\iso A^{\Phi}$.
\end{theorem}

\begin{proof}
The computation follows that of \autoref{thm:cancel-Cn}, applied to the layer map $\varphi(a_i\otimes c_r)=\varphi_r(a_i)\otimes c_{r+k}$. Here, the weights $c_r^2=\alpha_r c_{r+1}$ of $C$ enter the coefficient comparison, contributing an extra factor of $\alpha_r/\alpha_{r+k}$. Therefore, the resulting condition is that the scalars $\mu_{ij}$ 
from \autoref{def:k-alpha-iso} must be independent of
$r$
\autoref{def:k-alpha-iso}. This is equivalent to $\Phi$ being a
$(k;\alpha_0,\dots,\alpha_{m-1})$-isomorphism with $B\iso A^{\Phi}$.
\end{proof}

\begin{lemma}\label{lem:k-alpha-shift}
If $\Phi=(\varphi_0,\dots,\varphi_{m-1})$ is a
$(k;\alpha_0,\dots,\alpha_{m-1})$-isomorphism of $A$, then for each
$s\in\Z/m$ the cyclic shift $\Phi^{(s)}=(\varphi_s,\dots,\varphi_{s+m-1})$
is a $(k;\alpha_s,\dots,\alpha_{s+m-1})$-isomorphism of $A$.
\end{lemma}

\begin{proof}
As in \autoref{lem:cyclic-shift}, the defining condition constrains only
consecutive indices, which a cyclic shift preserves.
\end{proof}

Over an algebraically closed field the weights of a cyclic evolution algebra can be removed altogether.

\begin{lemma}\label{lem:weighted-trivial}
If $\K$ is algebraically closed, then
$C(\alpha_0,\dots,\alpha_{m-1})\iso C_m$ for all
$\alpha_0,\dots,\alpha_{m-1}\in\K^{\times}$.
\end{lemma}
\begin{proof}
Seek an isomorphism $\bar c_i\mapsto\gamma_i c_i$ from
$C(\alpha_0,\dots,\alpha_{m-1})$ (with $\bar c_i^2=\alpha_i\bar c_{i+1}$) to
$C_m$ (with $c_i^2=c_{i+1}$). Comparing
$\varphi(\bar c_i^2)=\alpha_i\gamma_{i+1}c_{i+1}$ with
$\varphi(\bar c_i)^2=\gamma_i^2 c_{i+1}$ forces the recursion
$\gamma_{i+1}=\gamma_i^2/\alpha_i$, whence
\[
  \gamma_i=\gamma_0^{\,2^{i}}\prod_{j=0}^{i-1}\alpha_j^{-2^{\,i-1-j}}.
\]
Closing the cycle ($\gamma_m=\gamma_0$, indices mod $m$) gives
$\gamma_0^{\,2^{m}-1}=\prod_{j=0}^{m-1}\alpha_j^{\,2^{\,m-1-j}}$, so it
suffices to take
\[
  \gamma_0=\Bigl(\textstyle\prod_{j=0}^{m-1}\alpha_j^{\,2^{\,m-1-j}}\Bigr)^{1/(2^{m}-1)}\in\K^{\times},
\]
which exists because $\K$ is algebraically closed. The remaining $\gamma_i$
are then nonzero, so $\varphi$ is an isomorphism.
\end{proof}
 
\begin{lemma}\label{lem:shifted-trivial}
Let $A,B$ be evolution $\K$-algebras and let $C=C(\alpha_0,\dots,\alpha_{m-1})$. If $\K$ is algebraically closed, then $A\otimes C\iso B\otimes C$ relative to $C$  shifting the second coordinate by $k$ if and only if $A\otimes C_m\iso B\otimes C_m$ relative to $C_m$.
\end{lemma}
\begin{proof}
Write $C=\bigoplus_{i=0}^{m-1} \K\bar c_i$ and $C_m=\bigoplus_{i=0}^{m-1} \K c_i$. Then, by \autoref{lem:weighted-trivial}, there exists an isomorphism in $\NEvol_\K$
$$\begin{array}{rcl}
\psi\colon C & \to & C_m\\
\bar c_i & \mapsto & \gamma_i c_i
\end{array}$$
Moreover, for any given $l\in\Z$, there exists an automorphism in 
$\NEvol_\K$
$$\begin{array}{rcl}
\Psi_l\colon C_m & \to & C_m\\
c_i & \mapsto & c_{i+l}\quad \text{(indices mod $m$)}
\end{array}$$
Therefore, $\phi\colon A\otimes C\to B\otimes C$ is an isomorphism relative to $C$ shifting the second coordinate by $k$ if and only if
$\widetilde{\phi}\colon A\otimes C_m\to B\otimes C_m,$
given by $$\widetilde{\phi}= \big(\mathrm{Id}_B\otimes(\Psi_{-k}\circ \psi)\big)\circ \phi\circ (\mathrm{Id}_A\otimes \psi^{-1})$$  is an isomorphism relative to $C_m$.
\end{proof}

\section{Local monomorphisms, transfer of isomorphism, and balance}\label{sec:cancellation_local}

In this section we show that the isomorphism type of tensor products in $\NEvol_\K$ is indeed controlled by local data of the common factor, and that non-cancellation phenomena appears when this common factor has higher balance.

We first define how we do compare the local data of two evolution algebras.

\begin{definition}\label{def:localmor}
Let $A=\bigoplus_{i\in I} \K a_i$ and $B=\bigoplus_{p\in P} \K b_p$ be  objects in $\NEvol_\K$  with associated structure matrices $(m_{ij})$ and $(\mathcal{M}_{pq})$ respectively, and let
$\varphi\colon A\to B$ be a linear map of the form
$\varphi(a_i)= b_{\sigma(i)}$ for some index map $\sigma\colon I\to P.$  We call $\varphi$ a \emph{local monomorphism} if
\[
  m_{ij}
   \ne 0\quad\text{implies}\quad \mathcal{M}_{\sigma(i)\sigma(j)}\ne 0
   \qquad(\text{for all}\ i,j).
\]
\end{definition}

\begin{remark}\label{rem:localgraph}
A local monomorphism $A\to B$ is exactly a morphism of 
graphs $\Gamma(A)\to\Gamma(B)$: the identity of \autoref{def:localmor} says precisely that
each edge $a_i\to a_j$ of $\Gamma(A)$ is matched by an edge
$b_{\sigma(i)}\to b_{\sigma(j)}$ of $\Gamma(B)$. The governing invariant of
such graphs is the balance of
\autoref{def:balance} and \autoref{def:graphbalance}; it coincides with the
invariant of Elduque--Labra \cite{ElduqueLabra2015,ElduqueLabra2019} and,
on the graph side, with that of Lov\'asz \cite[p.~154]{Lovasz}.
\end{remark}

We show that the isomorphism type of tensor products in  $\NEvol_\K$ is transferred by local monomorphisms.

\begin{theorem}\label{thm:transfer}
Suppose $A\otimes C\iso B\otimes C$ relative to $C$ in $\NEvol_\K$, and suppose there is a
local monomorphism $\Psi\colon D\to C$. Then $A\otimes D\iso B\otimes D$
relative to $D$ in $\NEvol_\K$.
\end{theorem}
\begin{proof}
Write $$A=\bigoplus_{i=1}^n \K a_i,\quad B=\bigoplus_{i=1}^n \K b_i,\quad C=\bigoplus_{i=1}^m \K c_i,\quad\text{and}\quad D=\bigoplus_{i=1}^l \K d_i,$$
and let $(m_{ij})$, $(\mathcal{M}_{ij})$, and $(\gamma_{ij})$ the associated structure matrices of $A$, $B$ and $C$. 

Let $\varphi(a_i\otimes c_r)=\lambda_{r,i}\,b_{\sigma_r(i)}\otimes c_r$ be
the given isomorphism relative to $C$. Then
\[
  \varphi\bigl((a_i\otimes c_r)^2\bigr)
   =\sum_{j,s} m_{ij}\,\gamma_{rs}\,\lambda_{s,j}\,
     b_{\sigma_{s}(j)}\otimes c_s,
\]
while
\[
  \varphi(a_i\otimes c_r)^2
   =\sum_{j,s} \lambda_{r,i}^2\,\gamma_{rs}\,
     \mathcal{M}_{\sigma_r(i),\,\sigma_s(j)}\,
     b_{\sigma_s(j)}\otimes c_s,
\]
so the two agree provided that
\begin{equation}\label{eq:proof_local_1}
  m_{ij}\,\lambda_{s,j}
   =\lambda_{r,i}^2\,\mathcal{M}_{\sigma_{r}(i),\,\sigma_{s}(j)}\quad\text{whenever $\gamma_{rs}\neq0$}.
\end{equation}

Let $\Psi(d)=c_{\rho(d)}$ be the local monomorphism, with index map $\rho$. Define
$$\begin{array}{rcl}
  \bar\varphi\colon A\otimes D & \longrightarrow & B\otimes D\\
  a_i\otimes d &\longmapsto & \lambda_{\rho(d),\,i}\,
    b_{\sigma_{\rho(d)}(i)}\otimes d
\end{array}$$
which is relative to $D$. Orthogonality is automatic: products of elements
with distinct second coordinates vanish on both sides, and for a fixed $d$
the map $\sigma_{\rho(d)}$ is injective, so $a_ia_j=0$ ($i\neq j$) is
preserved. For the squares, write $d^2=\sum_{d'}M_{dd'}d'$. Then
\[
  \bar\varphi\bigl((a_i\otimes d)^2\bigr)
   =\sum_{j,d'} m_{ij}\,M_{dd'}\,\lambda_{\rho(d'),j}\,
     b_{\sigma_{\rho(d')}(j)}\otimes d',
\]
while
\[
  \bar\varphi(a_i\otimes d)^2
   =\sum_{j,d'} \lambda_{\rho(d),i}^2\,M_{dd'}\,
     \mathcal{M}_{\sigma_{\rho(d)}(i),\,\sigma_{\rho(d')}(j)}\,
     b_{\sigma_{\rho(d')}(j)}\otimes d',
\]
so the two agree provided that 
\begin{equation}\label{eq:proof_local_2}
  m_{ij}\,\lambda_{\rho(d'),j}
   =\lambda_{\rho(d),i}^2\,\mathcal{M}_{\sigma_{\rho(d)}(i),\,\sigma_{\rho(d')}(j)}\quad\text{whenever $M_{dd'}\neq0$}.
\end{equation}
Now, if we denote $r=\rho(d)$, $s=\rho(d')$, then $\Psi$ being a local morphism gives that
$M_{dd'}\neq0$ implies $\gamma_{r,s}\neq0$, and therefore  \autoref{eq:proof_local_1} implies \autoref{eq:proof_local_2}. Hence
$\bar\varphi$ is an isomorphism relative to $D$.
\end{proof}

We can now connect the balance invariant to the zero-divisors in
$\G(\Evol_\K)$. Recall that $[A]$ is a \emph{zero-divisor} if
$[A]\bigl([Y]-[Z]\bigr)=0$ for some $Y\not\iso Z$, i.e.\ if
$A\otimes Y\iso A\otimes Z$ with $Y\not\iso Z$. The canonical example of a zero-divisor in $\G(\Evol_\K)$ is $[C_n]$ (\autoref{ex:non-cancel--cyclic}).

\begin{theorem}\label{thm:balance-prime}
Let $D$ be an indecomposable object in $\NEvol_\K$ with $\b(D)>1$, and let
$n$ be a divisor of $\b(D)$. If $A\otimes C_n\iso B\otimes C_n$ relative
to $C_n$, then $A\otimes D\iso B\otimes D$ relative to $D$.
\end{theorem}
\begin{proof}
Write $D=\bigoplus_{i=1}^m \K d_i$. Since $D$ is indecomposable and non-degenerate, the graph $\Gamma(D)$ is connected and has no sinks. Thus, according to \autoref{lem:balance_by_grading}, there exists a map $$h_n\colon\{d_i: i=1,\ldots, m\}\longrightarrow \Z/n,$$ such that if $(d_i,d_j)$ is an edge in $\Gamma(D)$ then $h_n(d_j)=h_n(d_i)+1 \pmod n.$
Setting
$$\begin{array}{rcl}
\Psi\colon D & \to & C_n\\
d_i & \mapsto & c_{h_n(d_i)}
\end{array}$$
defines a morphism of graphs
$\Gamma(D)\to\Gamma(C_n)$, and according to \autoref{rem:localgraph}, $\Psi$ is a local monomorphism. Thus
\autoref{thm:transfer} applies.
\end{proof}

\section{Cancellation when the balance is one}

The results above show that high balance forces zero-divisors. We now prove
the complementary statement: a balance-one indecomposable factor cancels.
The proof relies on a non-trivial translation of Lovász's combinatorial arguments \cite[pp.~153--155]{Lovasz} into the setting of evolution algebras, requiring several modifications to handle the algebraic structure.

\begin{theorem}\label{thm:cancel-balance-one}
Let $C$ be an indecomposable object in $\NEvol_\K$ with $\b(C)=1$. If
$A\otimes C\iso B\otimes C$ relative to $C$, then $A\iso B$.
\end{theorem}

\begin{proof}
Let $$A=\bigoplus_{i=1}^n \K a_i,\quad B=\bigoplus_{i=1}^n \K b_i,\quad\text{and}\quad C=\bigoplus_{i=1}^m \K c_i$$
and let $(\gamma_{ij})$ be the associated structure matrix of $C$. 

Given $\varphi\colon A\otimes C\iso B\otimes C$ relative to $C$, write
$\varphi(a_i\otimes c_r)=\varphi_r(a_i)\otimes c_r$, where
$$\begin{array}{rcl}
\varphi_r\colon A & \longrightarrow & B\\
a_i & \longmapsto & \lambda_{i,r}\,b_{\sigma_r(i)}
\end{array}$$
is a monomial linear isomorphism. Notice that each $\varphi_r$ sends lines to
lines, thus it preserves orthogonality, and it is an algebra isomorphism as soon
as it respects squares. 

Now, the identity
$\varphi((a_i\otimes c_r)^2)=\varphi(a_i\otimes c_r)^2$ expands to
$$\sum_s\gamma_{r,s}\,\varphi_s(a_i^2)\otimes c_s
   =\sum_s\gamma_{r,s}\,\varphi_r(a_i)^2\otimes c_s,$$ so that
\begin{equation}\label{eq:star-layer}
  (c_r,c_{r'})\in\Gamma(C)\ \Longrightarrow\
  \varphi_r(a_i)^2=\varphi_{r'}(a_i^2)\quad(\text{for all}\ i).
\end{equation}
Writing $j=\sigma_r(i)$, and taking into account that $\varphi_r$ and $\varphi_{r'}$ are linear isomorphism, \autoref{eq:star-layer} reads
\begin{equation}\label{eq:star-layer-inverse}
  (c_r,c_{r'})\in\Gamma(C)\ \Longrightarrow\
  \varphi_r^{-1}(b_j)^2=\varphi_{r'}^{-1}(b_j^2)\quad(\text{for all}\ j).
\end{equation}
\smallskip
\noindent\emph{Step 1 (composition along edges).} For any sequence of edges
$(c_{r_1},c_{r_1'}),\dots,(c_{r_{2k+1}},c_{r_{2k+1}'})$ in $\Gamma(C)$, applying \autoref{eq:star-layer} and \autoref{eq:star-layer-inverse} alternately yields: 
\begin{equation}\label{eq:compose}
  \bigl(\varphi_{r_{2k+1}}\varphi_{r_{2k}}^{-1}\cdots\varphi_{r_1}\bigr)(a)^2
=\bigl(\varphi_{r_{2k+1}'}\varphi_{r_{2k}'}^{-1}\cdots\varphi_{r_1'}\bigr)(a^2).
\end{equation}
\smallskip

\noindent\emph{Step 2 (a balance-one walk and its height function).} By \autoref{lem:balance-one-cycle} choose a closed walk $P=(x_0,x_1,\dots,x_n=x_0)$ in $\Gamma(C)$ with $\b(P)=1$, and extend it $n$-periodically, $x_\nu:=x_{\nu\bmod n}$ for $\nu\in\Z$. Put
$$
P_\nu =
\begin{cases}
(x_0) & \nu = 0,\\
(x_0, x_1, \dots, x_\nu) & \nu \ge 1,\\
(x_0, x_{-1}, \dots, x_\nu) & \nu < 0,
\end{cases}
$$ 
and $g(\nu)=\b(P_\nu)$, so $g(0)=0$. Each step $x_\nu\to x_{\nu+1}$ is a single edge of $\Gamma(C)$ traversed forwards or backwards, and one period contributes $\b(P)=1$; hence
\begin{equation}\label{eq:gprops}
  g(\nu+1)=g(\nu)\pm1,\qquad g(\nu+kn)=g(\nu)+k .
\end{equation}
\smallskip

\noindent\emph{Step 3 (the zeros of $g$ interlace).} From \autoref{eq:gprops},
$$g(kn)=k\ \text{ and }\ g(kn+l)\ge k-l\ \text{ for }\ 0\le l<n,$$ 
so $g(\nu)>0$ for $\nu\gg0$ and, symmetrically, $g(\nu)<0$ for $\nu\ll0$; thus $g$ has finitely many zeros. Call a zero $\nu$ \emph{ascending} if $g(\nu+1)=1$ and \emph{descending} if $g(\nu-1)=1$. Since $g$ moves by unit steps and drifts from $-\infty$ to $+\infty$, the ascending zeros $\nu_1<\dots<\nu_{p+1}$ and descending zeros $\mu_1<\dots<\mu_p$ interlace,
\begin{equation}\label{eq:zeros_interlace}
\nu_1<\mu_1\le\nu_2<\mu_2\le\dots<\mu_p\le\nu_{p+1},
\end{equation}
with exactly one more ascending than descending zero (a zero is counted in both lists, $\mu_i=\nu_{i+1}$, precisely when $g$ has a local valley there).
\smallskip

\noindent\emph{Step 4 (assembling the isomorphism).} At an ascending zero
$\nu_i$ the step $x_{\nu_i}\to x_{\nu_i+1}$ is a forward edge, and at a
descending zero $\mu_i$ the step $x_{\mu_i}\to x_{\mu_i-1}$ is a forward
edge, so all the relevant pairs lie in $\Gamma(C)$. Form the alternating
composites
\[
  \Phi=\varphi_{x_{\nu_1}}\varphi_{x_{\mu_1}}^{-1}\cdots
        \varphi_{x_{\mu_p}}^{-1}\varphi_{x_{\nu_{p+1}}},
  \qquad
  \Phi'=\varphi_{x_{\nu_1+1}}\varphi_{x_{\mu_1-1}}^{-1}\cdots
        \varphi_{x_{\mu_p-1}}^{-1}\varphi_{x_{\nu_{p+1}+1}},
\]
each a monomial linear isomorphism $A\to B$. By Step~1 (with $k=p$),
$\Phi(a)^2=\Phi'(a^2)$ for all $a$. It remains to check $\Phi=\Phi'$. We compare both alternating words to a common reduced form.

Call a zero $\nu$ of $g$ an \emph{upward crossing} if $g(\nu-1)=-1$ and
$g(\nu+1)=1$, and a \emph{downward crossing} if $g(\nu-1)=1$ and
$g(\nu+1)=-1$. Since $g$ moves by unit steps from $-\infty$ to $+\infty$
(Step~3), the upward crossings $e_1<\dots<e_{s+1}$ and downward crossings
$\pi_1<\dots<\pi_s$ strictly interlace,
\begin{equation}\label{eq:relevant_crosses}
  e_1<\pi_1<e_2<\dots<\pi_s<e_{s+1},
\end{equation}
with one more upward than downward crossing. An upward crossing is in
particular an ascending zero and a downward crossing a descending zero, so
$\{e_j\}\subseteq\{\nu_i\}$ and $\{\pi_j\}\subseteq\{\mu_i\}$; conversely an
ascending zero $\nu_i$ has $g(\nu_i-1)=-1$, hence is an upward crossing,
unless it coincides with the preceding descending zero, and symmetrically
for $\mu_i$, thus:
\begin{equation}\label{eq:cross-vs-zero}
  \nu_i\in\{e_1,\dots,e_{s+1}\}\iff \nu_i\neq\mu_{i-1},
  \qquad
  \mu_i\in\{\pi_1,\dots,\pi_s\}\iff \mu_i\neq\nu_{i+1}.
\end{equation}

When $\nu_i=\mu_{i-1}$ the two indices label the
same vertex, $x_{\nu_i}=x_{\mu_{i-1}}$, so the adjacent factors
$\varphi_{x_{\mu_{i-1}}}^{-1}\varphi_{x_{\nu_i}}=\mathrm{id}$ cancel in
$\Phi$. Deleting every such pair leaves precisely the factors indexed by the
crossings in \autoref{eq:relevant_crosses}:
\begin{equation}\label{eq:Phi-reduced}
  \Phi=\varphi_{x_{e_1}}\varphi_{x_{\pi_1}}^{-1}\cdots
       \varphi_{x_{\pi_s}}^{-1}\varphi_{x_{e_{s+1}}}=:\Phi''.
\end{equation}

Now, put $\nu_i':=\nu_i+1-n$ and
$\mu_i':=\mu_i-1-n$. By \autoref{eq:gprops},
\[
  g(\nu_i')=g(\nu_i+1)-1=0,\qquad g(\nu_i'-1)=g(\nu_i)-1=-1,
\]
\[
  g(\mu_i')=g(\mu_i-1)-1=0,\qquad g(\mu_i'+1)=g(\mu_i)-1=-1,
\]
and $n$-periodicity of $x_\bullet$ gives $x_{\nu_i'}=x_{\nu_i+1}$ and
$x_{\mu_i'}=x_{\mu_i-1}$, so
\[
  \Phi'=\varphi_{x_{\nu_1'}}\varphi_{x_{\mu_1'}}^{-1}\cdots
        \varphi_{x_{\mu_p'}}^{-1}\varphi_{x_{\nu_{p+1}'}}.
\]

The $\nu_i'$ are exactly
the zeros $\alpha$ of $g$ with $g(\alpha-1)=-1$: writing $\alpha=\nu+1-n$,
that is, $\nu=\alpha-1+n$, \autoref{eq:gprops} gives
$g(\nu)=g(\alpha-1)+1$ and $g(\nu+1)=g(\alpha)+1$, so
$\bigl(g(\alpha)=0,$ and $g(\alpha-1)=-1\bigr)$ holds simultaneously if and only if $\nu$ is an ascending
zero; hence there are $p+1$ such $\alpha$, namely the $\nu_i'$.
Symmetrically the $\mu_i'$ are all the zeros $\alpha$ with $g(\alpha+1)=-1$.
Exactly as in \autoref{eq:zeros_interlace} they interlace,
\[
  \nu_1'\le\mu_1'<\nu_2'\le\dots\le\mu_p'<\nu_{p+1}',
\]
and $\nu_i'\neq\mu_i'$ holds if and only if $\nu_i'$ is an upward crossing and $\mu_i'$
a downward crossing, that is, if and only if $\nu_i'\in\{e_j\}$ and $\mu_i'\in\{\pi_j\}$.

\emph{Reduction of $\Phi'$.} When $\nu_i'=\mu_i'$ the adjacent factors
$\varphi_{x_{\nu_i'}}\varphi_{x_{\mu_i'}}^{-1}=\mathrm{id}$ cancel in
$\Phi'$; deleting them leaves the same crossing-indexed word as in
\autoref{eq:Phi-reduced},
\[
  \Phi'=\varphi_{x_{e_1}}\varphi_{x_{\pi_1}}^{-1}\cdots
        \varphi_{x_{\pi_s}}^{-1}\varphi_{x_{e_{s+1}}}=\Phi''.
\]
Hence $\Phi=\Phi''=\Phi'$. Together with $\Phi(a)^2=\Phi'(a^2)$ from Step~1
this yields $\Phi(a)^2=\Phi(a^2)$ for every $a\in A$, so the monomial
bijection $\Phi$ preserves squares and is therefore an algebra isomorphism
$A\iso B$.
\end{proof}

\begin{remark}\label{rem:indec-resolved}
\autoref{thm:cancel-balance-one} is the converse, for indecomposable factors, of \autoref{thm:zero-divisors} under mild hypothesis: an indecomposable $C$ with
$\b(C)=1$ cancels relative to $C_m$, whereas $\b(C)>1$ makes it a zero-divisor. 
\end{remark}

\section{Indecomposability of Tensor Products}\label{sec:decomposition_tensor}

As noted in \autoref{sec:grothendieck}, understanding the ring generators of $\G(\Evol_\K)$ requires identifying when the tensor product of two indecomposable algebras remains indecomposable. We first show that the balance of the factors provides lower bounds for the number of summands in the optimal direct-sum decomposition into indecomposable evolution ideals \cite[Theorem 5.11]{CabreraSilesVelasco2016}.

\begin{theorem}\label{thm:ideals-tensor}
Let $X_1,X_2\in\Evol_\K$ be indecomposable, $d_i=\b(X_i)$, and $d=\gcd(d_1,d_2)$.
Then $$X_1\otimes X_2=\bigoplus_{c=0}^{d-1}I_c$$ where each $I_c$ is a non-trivial evolution ideal of $X_1\otimes X_2$.
\end{theorem}
\begin{proof}
Let $B_i$ be a natural basis of $X_i$, with associated structure matrix $(w^{(i)}_{rs})$, for $i=1,2$.
First, recall that $$\Gamma(X_1\otimes X_2,B_1\otimes B_2)=\Gamma(X_1,B_1)\times \Gamma(X_2,B_2).$$ 
Set $n_i=\dim(X_i)$. Since $d=\gcd(d_1,d_2)$, \autoref{lem:balance_by_grading} yields surjective maps $$h_i\colon\{1,\ldots,n_i\}\to \Z/d$$ such that, whenever $(r,s)\in\Gamma(X_i,B_i)$ one has $$h_i(s)=h_i(r)+1.$$ 

Define $H(r,s)=h_1(r)-h_2(s)\in\Z/d.$ Every edge
$\big((r,s),(k,l)\big)$ in $\Gamma(X_1\otimes X_2,B_1\otimes B_2)$ then satisfies
\[
  H(k,l)=\bigl(h_1(r)+1\bigr)-\bigl(h_2(s)+1\bigr)=H(r,s),
\]
so $H$ is constant on each connected component of the graph. For $c\in\Z/d$ set
\[
  I_c=\operatorname{span}_\K\{\,b_{1,r}\otimes b_{2,s}\in B_1\otimes B_2 : H(r,s)=c\,\}.
\]
For a basis vector in $I_c$,
$(b_{1,r}\otimes b_{2,s})^2=\sum_{k,l}w^{(1)}_{rk}w^{(2)}_{sl}\,(b_{1,k}\otimes b_{2,l})$, and every
contributing $(k,l)$ has $H(k,l)=c$, so the square lies in $I_c$. Moreover products
of distinct natural basis vectors vanish, so each $I_c$ is an evolution
ideal. Since $B_1\otimes B_2$ is partitioned by the values of $H$,
\[
  X_1\otimes X_2=\bigoplus_{c\in\Z/d} I_c,
\]
and surjectivity of maps $h(i)_d$ makes all $I_c$ non-empty.
\end{proof}

\begin{example}\label{lem:CnCm}
We can follow along the lines in the proof of \autoref{thm:ideals-tensor} to completely describe the tensor product of cyclic evolution algebras (compare
\cite[Eq.~(5)]{Lovasz}):
$$C_n\otimes C_m\iso C_{\lcm(n,m)}\otimes\K^{\gcd(n,m)}.$$
Equivalently,
$C_n\otimes C_m$ is a direct sum of $\gcd(n,m)$ copies of
$C_{\lcm(n,m)}$.

Indeed, $\b(C_n)=n$, $\b(C_m)=m$, and if $d=\gcd(n,m)$, then the maps $h_{i,d}$ are just the $\operatorname{mod} d$ reduction and therefore, for any $c\in\Z/d$ one has:
\begin{align*}
I_c &=\operatorname{span}_\K\{c_{kd+c}\otimes c_s : k=0,\ldots, \frac{n}{d}-1;\, s=0,\ldots,m-1;\, c=s\mod{d}\}\\
&\cong C_{\frac{n}{d} m}=C_{\lcm(n,m)}.
\end{align*}
Therefore
$$C_n\otimes C_m =\bigoplus_{c\in\Z/d} I_c \cong \bigoplus_{c\in\Z/d} C_{\lcm(n,m)} \iso C_{\lcm(n,m)}\otimes\K^{\gcd(n,m)}.$$
\end{example}

Notice that \autoref{thm:ideals-tensor} identifies a necessary condition for the tensor product to be indecomposable.

\begin{corollary}\label{cor:necessary_indecomposable}
Let $X_1,X_2\in\Evol_\K$ be indecomposable, $d_i=\b(X_i)$, and $d=\gcd(d_1,d_2)$. If $X_1\otimes X_2$ is indecomposable, then $d=1$.
\end{corollary}

The converse of \autoref{cor:necessary_indecomposable} holds, and we can now give a complete characterization of indecomposable tensor products. The following is a reinterpretation of \cite[Theorem 15]{ChenChen} in evolution-algebra form. 

\begin{theorem}
\label{thm:indecomp-tensor}
Let $X_1,X_2\in\Evol_\K$ be indecomposable, and $d_i=\b(X_i)$.
Then $X_1\otimes X_2$ is indecomposable if and only if
$\gcd(d_1,d_2)=1$, and in that case
$\b(X_1\otimes X_2)=d_1d_2$.
\end{theorem}
\begin{proof}
Recall that the algebra $X_i$ is non-degenerate. Hence, the graph $\Gamma(X_i,B_i)$ has no sinks and has non-trivial balance for any natural basis $B_i$ of $X_i$. Moreover, the notion of balance used here agrees with the notion of weight in \cite{ChenChen} (see \autoref{rm:balance=weight}). Since, for non-degenerate evolution algebras, indecomposability is equivalent to weak connectedness of the associated graph \cite{CabreraSilesVelasco2016}, the statement follows directly from \cite[Theorem 15]{ChenChen}.
\end{proof}

There is also a multiplicative description of indecomposability for
iterated tensor products, the evolution-algebra form of
\cite[Corollary 16]{ChenChen}, which we record here.

\begin{proposition}\label{prop:mfold}
Let $X_1,\dots,X_k\in\Evol_\K$ be indecomposable. Then $X_1\otimes\cdots\otimes X_k$ is indecomposable if
and only if the balances $b(X_1),\dots,b(X_k)$ are pairwise coprime, and
in that case
\[
  b\bigl(X_1\otimes\cdots\otimes X_k\bigr)=\prod_{i=1}^{k} b(X_i).
\]
\end{proposition}

\section{Zero-divisors among direct sums of cyclic algebras}

\autoref{conj:balance} predicts that an algebra of $\Evol_\K$ is a zero-divisor
in $\G(\Evol_\K)$ precisely when every one of its indecomposable direct
summands has balance strictly greater than $1$. The \emph{sufficiency}
direction holds in full generality: \autoref{thm:zero-divisors} produces a zero-divisor
relation for any direct sum $\bigoplus_j X_j$ all of whose summands have
balance $>1$, by combining the relations of the summands through
\autoref{lem:lovasz}. The \emph{necessity} direction is known only for indecomposable algebras (\autoref{rem:indec-resolved}, via the cancellation \autoref{thm:cancel-balance-one}), and its extension to reducible
algebras is a natural open problem. In this section we settle it for
arbitrary direct sums of cyclic algebras.

We now introduce the main object of study of this section.

\begin{definition}\label{def:cyclic-subring}
Let $\Cyc\subseteq\G(\Evol_\K)$ be the subring generated by the classes of the
cyclic algebras $\{[C_n]:n\ge 1\}$, where $C_1=\K$ is the unit.
\end{definition}

By \autoref{lem:CnCm} the product of two generators is a single generator up
to an integer factor,
\begin{equation}\label{eq:fusion}
  [C_n]\,[C_m]=\gcd(n,m)\,[C_{\lcm(n,m)}],\qquad [C_1]=1,
\end{equation}
so products of generators never leave the $\mathbb{Z}$-span of
$\{[C_n]\}$. As the $C_n$ are pairwise non-isomorphic indecomposables,
their classes are part of the free $\mathbb{Z}$-basis of $\G(\Evol_\K)$; hence
\[
  \Cyc=\bigoplus_{n\ge 1}\mathbb{Z}\,[C_n]
\]
is free abelian on $\{[C_n]\}$, with multiplication given by \autoref{eq:fusion}.

For a positive integer $F$, write $\tau(F)$ for its number of divisors, and define
\[
  \Cyc_F=\bigoplus_{d\mid F}\mathbb{Z}\,[C_d]\subseteq \Cyc.
\]
Since $d\mid F$ and $e\mid F$ imply that $\lcm(d,e)\mid F$, each $\Cyc_F$ is a
subring of $\Cyc$, free of rank $\tau(F)$, and $\Cyc=\bigcup_F \Cyc_F$.

We now construct \emph{balance characters} for $\Cyc$.

\begin{lemma}\label{lem:characters}
For each integer $m\ge 1$ there is a unique ring homomorphism
$\chi_m\colon \Cyc\to\mathbb{Z}$ with
\[
  \chi_m\bigl([C_n]\bigr)=
  \begin{cases} n, & n\mid m,\\[2pt] 0, & n\nmid m.\end{cases}
\]
\end{lemma}
\begin{proof}
Define $\chi_m$ on the basis by the formula given in the statement, and extend it $\mathbb{Z}$-linearly. Uniqueness is then immediate. It remains to verify that $\chi_m$ is a ring homomorphism.

Unitality is $\chi_m([C_1])=1$. For multiplicativity it suffices to check the
generators. Using \autoref{eq:fusion}, we obtain
\[
  \chi_m\bigl([C_n][C_{n'}]\bigr)
  =
  \begin{cases}
\gcd(n,n')\,\lcm(n,n'), &\text{if $\lcm(n,n')\mid m$,}\\
0, &\text{otherwise.}
  \end{cases}
\]
while
\[
  \chi_m\bigl([C_n]\bigr)\chi_m\bigl([C_{n'}]\bigr)
  =
  \begin{cases}
nn', &\text{if $n\mid m$ and $n'\mid m$,}\\
0, &\text{otherwise.}
  \end{cases}  
\]
These agree because $\gcd(n,n')\lcm(n,n')=nn'$, and $\lcm(n,n')\mid m$ if and only if $n\mid m$ and
$n'\mid m$.
\end{proof}

The character $\chi_m$ records the total dimension contributed by the cyclic summands whose balance divides $m$. In particular, $\chi_1([X])$ counts the copies of the unit algebra $\K$ occurring as summands of $X$. These characters diagonalize the cyclic subring after tensoring with $\mathbb{Q}$.

\begin{proposition}\label{prop:R-split}
For every $F\ge 1$, the map
$$\begin{array}{rcl}
\Phi_F\colon \Cyc_F\otimes\mathbb{Q} & \longrightarrow & \prod_{m\mid F}\mathbb{Q},\\
\xi & \longmapsto & \bigl(\chi_m(\xi)\bigr)_{m\mid F}
\end{array}$$
is an isomorphism of $\mathbb{Q}$-algebras. Consequently
$\Cyc_F\otimes\mathbb{Q}\cong\mathbb{Q}^{\tau(F)}$ is a product of fields,
and $\Cyc\otimes\mathbb{Q}$ is reduced.
\end{proposition}

\begin{proof}
Each coordinate of $\Phi_F$ is a ring homomorphism, so $\Phi_F$ is a
homomorphism between two $\mathbb{Q}$-algebras of the same dimension
$\tau(F)$. It therefore suffices to show $\Phi_F$ is injective, that is,
that the matrix of $\Phi_F$ in the basis $\{[C_d]:d\mid F\}$ of the
source and the coordinate basis of the target, namely
\[
  M=\bigl(\chi_m([C_d])\bigr)_{m\mid F,\;d\mid F}
\]
is invertible. List the divisors of $F$ in standard increasing numerical order. A nonzero entry $$M_{m,d}=\chi_m([C_d])=d$$ requires $d\mid m$, hence
$d\le m$, and therefore $M$ is lower triangular while its diagonal entries are
$M_{m,m}=m$. Thus
\[
  \det M=\prod_{d\mid F} d = F^\frac{\tau(F)}{2}\neq 0,
\]
and $\Phi_F$ is an isomorphism. Passing to the union over $F$, every
element of $\Cyc\otimes\mathbb{Q}$ lies in some $\Cyc_F\otimes\mathbb{Q}$,
which is reduced; hence $\Cyc\otimes\mathbb{Q}$ is reduced.
\end{proof}

Now we can characterize the zero-divisors in $\Cyc$.

\begin{corollary}\label{cor:zd-in-R}
An element $\alpha=\sum_{n} a_n[C_n]\in \Cyc$ is a zero-divisor in $\Cyc$ if
and only if there exists $m\ge 1$ such that
$$\chi_m(\alpha)=\sum_{n\mid m} a_n\,n=0.$$
\end{corollary}

\begin{proof}
Choose $F$ with $\alpha\in \Cyc_F$. As $\Cyc$ is free abelian, hence torsion-free, $\alpha$ is a zero-divisor in $\Cyc$ if and only if it is one in $\Cyc\otimes\mathbb{Q}$, and a witness may be taken in some $\Cyc_{F'}\otimes\mathbb{Q}$ with $F\mid F'$. Under the isomorphism $\Phi_{F'}$ of \autoref{prop:R-split}, zero-divisors of the product $\prod_{m\mid F'}\mathbb{Q}$ are exactly the tuples with a vanishing coordinate, so $\alpha$ is a zero-divisor if and only if $\chi_m(\alpha)=0$ for some $m\mid F'$. Finally, for any $m$ one has $\chi_m|_{\Cyc_F}=\chi_{\gcd(m,F)}$ because $n\mid F$ forces $n\mid m\Leftrightarrow n\mid\gcd(m,F)$; thus the existence of a vanishing coordinate is independent of the truncation, and is equivalent to $\chi_m(\alpha)=0$ for some $m\ge1$.
\end{proof}

\autoref{cor:zd-in-R} characterizes the zero-divisors of $\Cyc$ (inside $\Cyc$). The next lemma extends this characterization to the entire Grothendieck ring: a relation that does not occur within $\Cyc$ cannot arise after allowing the test element to range over arbitrary elements of $\G(\Evol_\K)$. The argument relies on the fact that elements of $\Cyc$ are integral over $\mathbb{Z}$, together with the torsion-freeness of $\G(\Evol_\K)$.

\begin{lemma}\label{lem:transfer}
Let $\alpha\in \Cyc$.
\begin{enumerate}[label={\rm (\arabic{*})}]
\item\label{lem:transfer_1} If $\chi_m(\alpha)\neq 0$ for every $m\ge 1$, then $\alpha$ is a
  non-zero-divisor in $\G(\Evol_\K)$.
\item\label{lem:transfer_2} If $\chi_m(\alpha)=0$ for some $m\ge 1$, then $\alpha$ is a
  zero-divisor in $\G(\Evol_\K)$, with a witness already in $\Cyc$.
\end{enumerate}
\end{lemma}

\begin{proof}
Fix $F$ with $\alpha\in \Cyc_F$.

\ref{lem:transfer_1}: If no character vanishes on $\alpha$, then $\Phi_F(\alpha)$ has all
coordinates nonzero, so $\alpha$ is a unit in
$\Cyc_F\otimes\mathbb{Q}\cong\mathbb{Q}^{\tau(F)}$. Let
$\beta\in \Cyc_F\otimes\mathbb{Q}$ be its inverse and clear denominators:
there exist an integer $c\ge 1$ and an element $\delta=c\beta\in \Cyc_F$
with
\[
  \alpha\,\delta=c\quad\text{in}\quad \Cyc_F\subseteq\G(\Evol_\K) .
\]
Now suppose $\alpha W=0$ for some $W\in\G(\Evol_\K)$. Then
\[
  cW=(\alpha\delta)W=\delta(\alpha W)=0,
\]
and since $\G(\Evol_\K)$ is free abelian, hence torsion-free, and $c\neq 0$, we
get $W=0$. Thus $\alpha$ is a non-zero-divisor in $\G(\Evol_\K)$.

\ref{lem:transfer_2}: If $\chi_{m_0}(\alpha)=0$, then $\Phi_F(\alpha)$ has a vanishing
coordinate, so $\alpha$ is a zero-divisor in
$\Cyc_F\otimes\mathbb{Q}$. Clearing denominators in an identity
$\alpha\eta_0=0$ with $0\neq\eta_0\in \Cyc_F\otimes\mathbb{Q}$ yields
$0\neq\eta\in \Cyc_F$ with $\alpha\eta=0$ in $\Cyc_F\subseteq\G(\Evol_\K)$.
As $\eta\neq 0$ in $\G(\Evol_\K)$, $\alpha$ is a zero-divisor there.
\end{proof}

The key point in the proof of \autoref{lem:transfer}.\ref{lem:transfer_1} is that the test element $W$ may be an arbitrary element of $\G(\Evol_\K)$: it need not be cyclic, nor even a direct sum of indecomposable elements of bounded balance. Nevertheless, the relation
$$\alpha\delta=c\cdot 1,$$ 
which holds in $\Cyc$, implies that $cW=0$.
This is illustrated in the following example.

\begin{example}
Let $\alpha=[\K\oplus C_2]=1+[C_2]$. Then, taking $\delta=3[\K]-[C_2]=3-[C_2],$ one has
\[
  \bigl(1+[C_2]\bigr)\bigl(3-[C_2]\bigr)=3,
\]
since $[C_2]^2=[C_2\otimes C_2]=[C_2\oplus C_2]=2[C_2]$. So any $W$ satisfying
$(1+[C_2])W=0$ also satisfies $3W=0$ and hence $W=0$. Therefore, $[\K\oplus C_2]$  is not a zero-divisor in $\G(\Evol_\K)$, even though its summand $C_2$ itself is a
zero-divisor.
\end{example}

Finally, we can describe the non-degenerate algebras whose optimal direct-sum decomposition into indecomposable evolution ideals involve just cyclic algebras, and give rise to zero-divisors in $\G(\Evol_\K)$.

\begin{theorem}\label{thm:cyclic-conjecture}
Let $X=\bigoplus_{i=1}^{r}C_{n_i}$ be a finite direct sum of cyclic
algebras, with $n_i\ge 1$. Then $[X]$ is a zero-divisor in $\G(\Evol_\K)$ if
and only if $n_i>1$ for every $i$; equivalently, if and only if every
indecomposable direct summand of $X$ has balance greater than $1$. In
particular, \autoref{conj:balance} holds for all direct sums of cyclic algebras.
\end{theorem}

\begin{proof}
The indecomposable summands of $X$ are the cyclic algebras $C_{n_i}$,
which have balance $b(C_{n_i})=n_i$; so the two formulations of the
condition coincide. Write $\alpha=[X]=\sum_{i}[C_{n_i}]\in \Cyc$, so that
for every $m\ge 1$
\[
  \chi_m(\alpha)=\sum_{\{i\,:\,n_i\mid m\}} n_i\;\ge 0 ,
\]
a sum of positive integers, which vanishes exactly when no $n_i$ divides
$m$.

If some $n_{i_0}=1$, then $n_{i_0}\mid m$ for every $m$, so
$\chi_m(\alpha)\ge 1>0$ for all $m$. Therefore, by \autoref{lem:transfer}.\ref{lem:transfer_1},
$\alpha$ is a non-zero-divisor in $\G(\Evol_\K)$.

If instead $n_i>1$ for all $i$, then no $n_i$ divides $m=1$, so
$\chi_1(\alpha)=0$; by \autoref{lem:transfer}.\ref{lem:transfer_2}, $\alpha$ is a
zero-divisor in $\G(\Evol_\K)$, with a witness in $\Cyc$.
\end{proof}

\begin{remark}\label{rmk:correction}
Observe that \autoref{thm:cyclic-conjecture} clarifies the statement of \autoref{rem:reduce-indec}. Indeed,
\autoref{thm:cyclic-conjecture} handles the necessity direction when all summands are cyclic, without invoking the reduction given by \autoref{lem:lovasz}. Instead, it works directly in $\G(\Evol_\K)$ via the balance characters of $\Cyc$ introduced in \autoref{lem:characters}.
\end{remark}


\begin{thebibliography}{99}

\bibitem{BoudiCabreraSiles2022}
N.~Boudi, Y.~Cabrera Casado, and M.~Siles Molina,
\emph{Natural families in evolution algebras},
Publ.\ Mat.\ \textbf{66} (2022), no.~1, 159--181.

\bibitem{CabreraMartinMartinTocino2023}
Y.~Cabrera Casado, D.~Mart\'in Barquero, C.~Mart\'in Gonz\'alez, A.~Tocino,
\emph{Tensor product of evolution algebras},
Mediterr. J. Math. \textbf{20} (2023), no.~1, Paper No.~43.

\bibitem{CabreraSilesVelasco2016}
Y.~Cabrera Casado, M.~Siles Molina, M.~V.~Velasco,
\emph{Evolution algebras of arbitrary dimension and their decompositions},
Linear Algebra Appl. \textbf{495} (2016), 122--162.

\bibitem{CabreraSilesVelasco2017}
Y.~Cabrera Casado, M.~Siles Molina, M.~V.~Velasco,
\emph{Classification of three-dimensional evolution algebras},
Linear Algebra Appl. \textbf{524} (2017), 68--108.

\bibitem{ChenChen}
S.~Chen, X.~Chen,
\emph{Weak connectedness of tensor product of digraphs},
Discrete Appl. Math. \textbf{185} (2015), 52--58.

\bibitem{CostoyaLigourasTocinoViruel2022}
C.~Costoya, P.~Ligouras, A.~Tocino, A.~Viruel,
\emph{Regular evolution algebras are universally finite},
Proc. Amer. Math. Soc. \textbf{150} (2022), no.~3, 919--925.

\bibitem{CostoyaMunozTocinoViruel2023}
C.~Costoya, V.~Mu\~noz, A.~Tocino, A.~Viruel,
\emph{Automorphism groups of Cayley evolution algebras},
Rev. R. Acad. Cienc. Exactas F\'is. Nat. Ser. A Mat. RACSAM
\textbf{117} (2023), no.~2, Paper No.~82.

\bibitem{ElduqueLabra2015}
A.~Elduque, A.~Labra,
\emph{Evolution algebras and graphs},
J. Algebra Appl. \textbf{14} (2015), no.~7, 1550103.

\bibitem{ElduqueLabra2019}
A.~Elduque, A.~Labra,
\emph{Evolution algebras, automorphisms, and graphs},
Linear Multilinear Algebra \textbf{69} (2021), no.~1, 1--12.



\bibitem{Lovasz}
L.~Lov\'asz,
\emph{On the cancellation law among finite relational structures},
Period. Math. Hungar. \textbf{1} (1971), no.~2, 145--156.

\bibitem{Tian2008}
J.~P.~Tian,
\emph{Evolution Algebras and their Applications},
Lecture Notes in Mathematics \textbf{1921}, Springer, Berlin, 2008.

\bibitem{TianVojtechovsky2006}
J.~P.~Tian, P.~Vojt\v{e}chovsk\'y,
\emph{Mathematical concepts of evolution algebras in non-Mendelian
genetics}, Quasigroups Related Systems \textbf{14} (2006), no.~1,
111--122.

\end{thebibliography}
\end{document}